\documentclass[a4, 12pt]{article}
  \usepackage{amsmath, amssymb, latexsym, amscd, amsthm,amstext}
 \usepackage{authblk}
  \usepackage{hyperref}                            % For creating hyperlinks in cross references  
  \usepackage{xcolor}
  \usepackage{times}
  \usepackage{enumerate}  
\usepackage{anysize}
\marginsize{1.5in}{1.3in}{0.5in}{0.5in}

\renewcommand\appendix{\par
	\setcounter{section}{0}
	\setcounter{subsection}{0}
	\setcounter{figure}{0}
	\setcounter{table}{0}
	\renewcommand\thesection{Appendix \Alph{section}}
	\renewcommand\thefigure{\Alph{section}\arabic{figure}}
	\renewcommand\thetable{\Alph{section}\arabic{table}}
} 
  
  \def\R{\mathbb{ R}}
  \def\C{\mathbb{ C}}

  \def\Sy{\mathbb{ S}}
  \def\H{\mathbb{ H}}

  \def\F{\mathbb{ F}}

  \def\bP{\mathbf{ P}} 
     
  \def\bD{\mathbf{ D}}  
  \def\bA{\mathbf{ A}}  
      
  \def\bC{\mathbf{ C}}    
  \def\bX{\mathbf{ X}}
  \def\bY{\mathbf{ Y}}

  \def\rk{\mbox{\rm rank}}
  \def\tr{\mbox{\rm Tr}}
  
  \def\diag{\mbox{\rm diag}}

  \newtheorem{theo}{\bf Theorem}%[section]
  \newtheorem{coro}[theo]{\bf Corollary}
  \newtheorem{lemm}[theo]{\bf Lemma}

\theoremstyle{definition}
 \newtheorem{exam}{\it Example}
\newtheorem{rema}{\it Remark}
\newtheorem{algo}{\bf Algorithm}

  \rmfamily
  
\begin{document}
 \title{\bf 
Simultaneous diagonalization via %$^*$-
congruence of Hermitian matrices: 
some equivalent conditions and 
a numerical solution
}

\author{Hieu Le Thanh\thanks{%Corresponding author
%%The author is supported by 
%%Vietnam National Foundation for Science and Technology Development (NAFOSTED) under grant number 
%%		101.01-2017.30. 
%
Email: \texttt{lethanhhieu@qnu.edu.vn}}
}
%\affil[a]{\small Department of Mathematics and Statistic, Quy Nhon University, Vietnam}

\author{Ngan Nguyen Thi\thanks{
%		The author is partially supported by
%  		the Research Council K.U.Leuven,
%  		project
%  		OT/10/038 (Multi-parameter model order reduction and its applications),
%  		CoE EF/05/006 Optimization in  Engineering (OPTEC),
%   		and by
%   		the Interuniversity Attraction Poles Programme, initiated by the Belgian State,  Science Policy Office,
%   		Belgian Network DYSCO (Dynamical Systems, Control, and Optimization).
%
Email: \texttt{nguyenthingan@ttn.edu.vn} }
}
%\affil[b]{Department of Mathematics, Tay Nguyen University, Vietnam}
 \affil{\small Department of Mathematics and Statistics, Quy Nhon University, Vietnam}

 \maketitle
 
 \begin{abstract}
 This paper aims at solving the Hermitian SDC problem, i.e.,
  that of \textit{simultaneously diagonalizing via $*$-congruence} a collection of finitely many (not need pairwise commute) Hermitian matrices.
Theoretically,
we provide some equivalent conditions for that such a matrix collection can be simultaneously diagonalized via $^*$-congruence.% by a nonsingular matrix.
Interestingly, one of such conditions
leads to the existence of a positive definite solution to a
semidefinite program (SDP).
From practical point of view, 
we propose an algorithm for numerically solving such problem.
The proposed algorithm is a combination of
 (1) a positive semidefinite program detecting whether the initial Hermitian matrices are simultaneously diagonalizable via $*$-congruence, and 
 (2) a Jacobi-like algorithm for simultaneously diagonalizing via $*$-congruence  the commuting normal matrices derived from the previous stage. 
Illustrating examples by hand/coding in \textsc{Matlab} are also presented.
 \end{abstract}
 

%\tableofcontents

% \linenumbers


\section{Introduction}\label{rmp_intr}

\textbf{Notations and preliminaries}.
Let $\F$ denote the field of real numbers $\R$ or complex ones $\C,$
and
$\F^{n\times n}$
denote
the set of all $n\times n$ matrices
with entries in $\F.$
Let 
$\Sy^n $
(resp., $\H^n$)
denote
the set of real symmetric (resp., Hermitian) matrices in $\F^{n\times n}.$ 
By $.^*$ we denote the  conjugate transpose of a matrix.
For $A \in \H^n,$ 
by
$A\succeq 0$ 
(resp., $A\succ 0$)
we mean $A$ is positive semidefinite (resp., positive definite)
as usual.

From now on, without stated otherwise,
 $C_1, \ldots, C_m \in \F^{n\times n}$ are always Hermitian matrices.
 By
$\bC(\lambda)$ we denote the \textit{Hermitian pencil}
$\bC(\lambda) =: \sum_{i=1}^m \lambda_i C_i,$
i.e., 
the parameters 
$\lambda_1, \ldots, \lambda_m$ ranges only  over real numbers.
Finally,
Hermitian matrices 
$
C_1, \ldots, C_m
$
are said to be
\begin{enumerate}[(i)]
	\item 
	\textit{simultaneously diagonalizable via similarity} on $\F,$
		abbreviated $\F$-\textit{SDS},
	if there exists a nonsingular matrix $P\in \F^{n\times n}$ such that
	$P^{-1} C_iP$'s are all diagonal matrices in $\F^{n \times n}.$
	When $m=1,$ we will say  ''$C_1$ is \textit{similar} to a diagonal matrix'' 
	or ``$C_1$ is diagonalizable (via similarity)''
	as usual.
	
	\item 
	\textit{simultaneously diagonalizable via $*$-congruence} on $\F$,
	abbreviated $\F$-\textit{SDC}, 
	if 
	there exists a nonsingular matrix $P\in \F^{n\times n}$ such that
	$P^* C_iP$ is  diagonal for every $i=1, \ldots, m.$
	Here $P^*$ means the conjugate transpose of $P.$
	When $m=1,$ we will say ``$C_1$ is congruent to a diagonal matrix'' as usual.
	
	It is worth mentioning that every diagonal matrix 
	 $P^* C_iP$ is always real since $C_i$ is Hermitian.
	Moreover,
	we will see 
	in Theorems \ref{theo:maxrk} and \ref{theo:H2} below that
	$P$ can be chosen to be real if $C_i$'s are all real.
	Additionally,
	one should distinguish that $*$-congruence in this paper is different from  $\mathrm{T}$-one in \cite{Bustamante2020}.
	These two types of congruence will coincide only when the Hermitian matrices are real.
	
\item
	\textit{commuting} if they pairwise commute, i.e., $C_iC_j = C_j C_i$ for every $i,j=1, \ldots, m.$
%
%	\item 
%	\textit{simultaneously diagonalizable via $\mathrm{T}$-congruence} on $\F,$ shortly said $\F$-\textit{SDC}, 
%	if they are symmetric in $\F^{n\times n}$ and there exists a nonsingular matrix $P\in \F^{n\times n}$ such that
%	$P^\mathrm{T} C_iP$ is diagonal in $\F^{n \times n}$ for every $i=1, \ldots,m.$		
\end{enumerate}
%

% if not state otherwise,
%we will shortly say ``via congruence'' instead of ``via $*$-congruence''.
%And, 
Depend upon particular situation in the rest of this paper, 
the term ``SDC'' will mean
either
``\textit{simultaneous diagnalization}''
or
``\textit{simultaneously diagnalizing}'',
or
``\textit{simultaneously diagnalizable}''  via $*$-congruence.
It is analogously to the term ``\textit{SDS}''.

 We now recall some well-know results that
will be frequently used in the rest of the paper.

\begin{lemm} \label{lem:H1}
	Let $A,B\in \F^{n\times n}$ be 
	$A =  \diag(\alpha_1 I_{n_1}, \ldots, \alpha_k I_{n_k}) ,$
	$\F\ni \alpha_i$'s are distinct.
	If $AB=BA$ 
	then 
	$B= \diag(B_1, \ldots, B_k)$
	with 
	$B_i \in \F^{n_i \times n_i}$ for every $i=1, \ldots, k.$
	Furthermore,  
	$B$ is Hermitian (resp., symmetric) if and only if so are 
	$B_i$'s. 
\end{lemm}
\begin{proof}
	%	a)
	Partitioning $B$ as
	$B = [B_{ij}]_{i,j=1,\ldots, k},$
	where %$B_{ji} = B_{ij}^*$ for all $i,j$ and   
	$B_{ii}$ is a square submatrix of size $n_i\times n_i,$ $i=1, \ldots,k$
	and off-diagonal blocks are of appropriate sizes.
	It then follows from
	$$
	\begin{bmatrix}
	\alpha_1 B_{11} & \ldots & \alpha_1 B_{1k}\\
	\vdots          & \ddots & \vdots \\
	\alpha_k B_{k1} & \ldots & \alpha_k B_{kk}	
	\end{bmatrix}
	= AB = BA =
	\begin{bmatrix}
	\alpha_1 B_{11} & \ldots & \alpha_k B_{1k}\\
	\vdots          & \ddots & \vdots \\
	\alpha_1 B_{k1} & \ldots & \alpha_k B_{kk}	
	\end{bmatrix}
	$$
	that
	$\alpha_i B_{ij} = \alpha_j B_{ij} , \enskip \forall i\neq j.$
	Thus $B_{ij}=0$ for every $i\neq j.$
	
	The last claim is trivial.
\end{proof}

\begin{lemm} \label{lemm:HB1} (See, eg., in \cite{b398}) %Horn's Book
	(i)
	Every $A\in \H^n$ can be diagonalized via similarity by a unitary matrix.
	That is, it can be written as
	$A = U \Lambda U^*,$
	where $U$ is unitary, $\Lambda$ is real diagonal and
	is uniquely defined up to a permutation of diagonal  elements. 
	
	Moreover, if $A \in \Sy^n$ then $U$ can be picked to be real.
	
	(ii) %Thm 1.3.19
	Let  $C_1, \ldots, C_m \in \F^{n\times n}$ such that each of which is similar to a diagonal matrix.
	They are then $\F$-SDS if and only if they are commuting.
	%	\red{dung cho ma tran PHUC theo Horn nhung co dung cho THUC?}
	
	(iii) %Lem 1.3.10	
	Let $A \in \F^{n\times n},$ $B \in \F^{m\times m}.$
	The matrix $M = \diag(A, B)$ 
	is diagonalizable via similarity if and only if so are both $A$ and $B.$
	%	\red{Boi ma tran kg suy bien, ma tran duong cheo THUC hay PHUC?}
	
	(iv)
	A complex symmetric matrix $A$ is diagonalizable via similarity
%	 i.e., $P^{-1}A P$ is diagonal for some invertible matrix $P\in \C^{n\times n},$ 
	if and only if it is complex orthogonally diagonalizable, i.e.,   
	$Q^{-1}A Q$
	is diagonal for some complex orthogonal matrix $Q \in \C^{n\times n}:$ $Q^T Q = I.$ 
\end{lemm}
\begin{proof}
	The first and last parts can be found in \cite{b398}.
	We prove  (ii) and (iii) for real matrices, the complex setting was proved in \cite{b398}.
	
	(ii) Suppose $C_1,  \ldots, C_m \in \R^{n\times n}$ are commuting. We prove by induction on $n.$
	
	For $n=1,$ there are nothing to prove.
	Suppose that $n\geq 2$ and that the assertion has been proved for all collections of $m$ commuting real matrices of size $k\times k,$ 
	$1\leq k \leq n-1.$

	By the hypothesis, 
	$C_1$ is diagonalized via similarity by a nonsingular matrix $P \in \R^{n\times n},$
	$$
	P^{-1} C_1 P = \diag(\alpha_1 I_{n_1}, \ldots, \alpha_k I_{n_k}),
	\enskip
	n_1 +\ldots +n_k =n, \enskip 
	\alpha_i \neq \alpha_j, \forall i\neq j.
	$$
	The commutativity of $C_1, C_i$  implies that of 
	$P^{-1} C_1 P$ and $P^{-1} C_i P.$
	By Lemma \ref{lem:H1},
	$$
	P^{-1} C_i P = \diag(C_{i1}, \ldots, C_{ik}),
	\enskip
	C_{it} \in \R^{n_t \times n_t}, \forall t =1, \ldots, k.
	$$
	Furthermore, 
	$P^{-1} C_i P$ and $P^{-1} C_j P$ commute due to the commutativity of  $C_i, C_j.$
	So, 
	$C_{it} $ and $C_{jt}$ commute
	for every $t=1, \ldots, t.$
	This means the matrices 
	$C_{2t},$ $\ldots,$ $C_{mt}$
	are commuting for every $t=1, \ldots,k.$
	By the induction hypothesis, for each $t=1, \ldots,k,$
	the matrices
	$C_{2t}, \ldots, C_{mt}$
	are $\R$-SDS by a nonsingular matrix $Q_t \in \R^{n_t \times n_t},$
	and so are 
	$I_{n_t}, C_{2t}, \ldots, C_{mt}.$
	Let 
	$$
	Q = \diag(Q_1, \ldots, Q_k) \in \R^{n \times n},
	$$
	and set
	$U = PQ.$ 
	One can check $C_1, \ldots, C_m$ are $\R$-SDS by $U.$
	
	The converse is trivial.
	
	(iii) Suppose $M$ is DS in $\R^{(n+m)\times(n+m)}.$
	Then there exists a nonsingular matrix $S \in \R^{(n+m)\times(n+m)}$ such that
	$$
	S^{-1} M S = \Lambda = \diag(\alpha_1, \ldots, \alpha_n, \ldots
	 \alpha_{n+m}), \enskip 
	\alpha_i  \in \R, \enskip 
	i=1, \ldots,m+n.
	$$
	Let $s_j$ be the $j$-th column of $S$ with
	$$
	s_j 
	=
	\begin{bmatrix}
	\xi_j \\  \nu_j
	\end{bmatrix}, \enskip
	\xi_j \in \R^n,  \nu_j \in \R^m.
	$$
	Let 
	$E = [\xi_1 \ldots \xi_{n+m}] \in \R^{n \times (n+m)},$
	$N = [\nu_1 \ldots \nu_{n+m}] \in \R^{m \times (n+m)}.$
	If $\rk (E) < n$ or $\rk(N) <m$ then
	$$
	m+n = \rk (S) \leq \rk(E) +\rk(N) < m+n, 
	$$
	which is impossible. 
	Thus $\rk(E)=n$ and $\rk(N)= m.$
	Moreover,
	it follows from $MS = S\Lambda$ that  
	the $j$th column
	$\xi_j $  of $E$
	is the eigenvector of $A$ w.r.t the eigenvalue $\lambda_j.$
	This means $A$ has $n$ linearly independent eigenvectors.
	By the same reasoning,  $B$  also has $m$ linearly independent eigenvectors.
	Thus $A$ and $B$ are diagonalizable via similarity in $\R^{n\times n}$ and $\R^{m\times m},$ respectively.
\end{proof}

\noindent
\textbf{History of SDC problems}.
Matrix simultaneous  diagonalizations,  via similarity or congruence,
appear in a number of research areas, for examples,
quadratic equations and optimization \cite{HU2007, jiang2016}, multi-linear algebra \cite{r390},
signal processing, data analysis \cite{r390}, quantum mechanics \cite{Shankar1994}, \ldots
%
%It is shown in, eg., \cite{b398} that 
%there is a relationship of two concepts of simultaneous diagonalizations via similarity and via congruence (see Lemma \ref{lemm:HB1} below).
%One should hence need to distinguish these two concepts.

The SDC problem, i.e.,
the one of finding a nonsingular matrix that simultaneously diagonalizes a collection of matrices via congruence,  
 can be dated back to 1868 by Weierstrass
\cite{Weierstrass1868}, 
1930s by Finsler \cite{Finsler37, Albert38, Hestenes40},
and later studies developed some conditions ensuring a collection of quadratic forms are SDC (see, eg., in \cite{More93, Pong2014} and references therein).
However, 
almost
these works provide only sufficient (but not necessary) conditions, 
except for, eg.,
that in
 \cite{Uhlig73}, \cite{b398, b172} and references therein
 dealing with 
 two Hermitian/symmetric matrices;
or that in \cite{jiang2016} focusing on
real symmetric matrices which have a positive definite linear combination. 
Some other works focus on simultaneous block-diagonalizations of normal matrices, so does of Hermitian ones \cite{b398, Watters1974}.
An (equivalently) algorithmic condition for the real symmetric SDC problem
is given in a preprint manuscript \cite{Ngan2020}.
Very recently in the seminal paper \cite{Bustamante2020},
the authors develop
an equivalent condition solving
the \textit{complex symmetric} SDC problem.
Note that this admits $\mathrm{T}$-congruence, 
since the matrices there are complex symmetric,
instead of $*$-one as in this paper.
The simultaneous diagonalization via $\mathrm{T}$-congruence for complex symmetric matrices
\textit{does not guarantee} for real setting.
That is, the resulting nonsingular and diagonal matrices may not be real even thought the initial ones are real.
An example
given in  \cite{Bustamante2020}
is
that
$$
C_1 = \begin{bmatrix} 0 & 1\\ 1 & 1\end{bmatrix}, \enskip
C_2 = \begin{bmatrix} 1 & 1\\ 1 & 0\end{bmatrix}
\in \Sy^2
$$
are SDC via $\mathrm{T}$-congruence.
However,
the resulting nonsingular matrix $P$
and
the corresponding diagonal ones
$P^TC_1P, P^TC_2P $
are properly complex.

We will see in this paper that the simultaneous diagonalizableness via $*$-congruence  of Hermitian matrices will immediately holds true for real symmetric setting.
Turning to the two matrices above,
we will see from
Theorem \ref{theo:maxrk} below
that
they are not $\R$-SDC because
$
C_1^{-1} C_2 = 
\begin{bmatrix}
0 & -1\\ 1 & 1
\end{bmatrix}
$
is not similar to a real diagonal matrix 
(it has only complex eigenvalues $\frac{1\pm i\sqrt{3} }{2}$). 

%As shown in 
% \cite[Example 13]{Bustamante2020}  (see also, Example \ref{ex:1})
% that
% there is a collection of 
%  real symmetric matrices are only $\C$-SDC via $\mathrm{T}$-congruence
%  but not $\R$-SDC via $\mathrm{T}$-one (identically as $*$-congruence).
% %
%For the real symmetric setting,

From the practical point of view, there is 
a few works dealing with 
numerical methods for solving SDC problems with respect to particular type of matrices.
The work in  
 \cite{BuByMe93}
 deal with 
two commuting  normal matrices 
 where they apply Jacobi-like algorithm.
This method is then extended to the case of more than commuting normal matrices, and is performed in \textsc{Matlab}
very recently \cite{Chris2020}.  

\noindent
 \textbf{Contribution of the paper}. 
 In this paper,
 we completely solve the Hermitian SDC problem.
The paper contains the following contributions:
\begin{enumerate}[$\bullet$]	
\item 
	We develop some sufficient and necessary conditions,
	see Theorems \ref{theo:H1}, \ref{theo:maxrk} and \ref{theo:H2},
	 for that
	a collection of finitely many Hermitian matrices can be 
	simultaneously diagonalized via $*$-congruence.
	As a consequence, this
	solves the  long-standing SDC problem for real symmetric matrices mentioned as an
	open problem in \cite{HU2007}.

\item
	Interestingly,
	one of such the conditions in Theorem \ref{theo:H2} requires the existence of a positive definite solution to a semidefinite program.
	This helps us to check whether a collection of Hermitian matrices is SDC or not. 
	In case the initial matrices are SDC,
	we apply the existing Jacobi-like method in \cite{BuByMe93, Chris2020} (for simultaneously diagonalizing a collection of commuting normal matrices by a unitary matrix)
	to  simultaneously diagonalize the commuting Hermitian matrices from the previous stage.
	The Hermitian SDC problem is hence completely solved.

\item
	In line of giving an equivalent condition that requires the maximum rank of the Hermitian pencil (Theorem \ref{theo:maxrk}), 
	we propose a Schm\"{u}dgen-like procedure for finding such the maximum rank in Algorithm \ref{alg:Schm}.
	This may be applied in some other simultaneous diagonalizations, for example that in \cite{Bustamante2020}.

\item
The corresponding algorithms are also presented and implemented in \textsc{Matlab} among which the main is Algorithm \ref{alg:sdc}.	
\end{enumerate}

%====

% [J-B. Hiriart-Urruty, Potpourri of conjectures and open
%questions in nonlinear analysis and optimization, SIAM Review 49(2),
%2007].
%====

Unlike to the complex symmetric matrices as \cite{Bustamante2020},
what discussed for Hermitian setting in this paper thus immediately imply to real symmetric one.

\textbf{Construction of the paper}. 
In Section \ref{sec:sdc} we give a comprehensive description on Hermitian SDC property.
A relationship between the SDC and SDS problems is included as well. 
In addition,
since the main result in this section, Theorem \ref{theo:maxrk}, 
asks to find a maximum rank linear combination of the initial matrices,
we suggest 
a method responding to this requirement in Subsection \ref{sec:schm}.
Section \ref{sec:sdp} presents some other SDC equivalent conditions among which leads to use semidifinite programs for detecting the simultaneous diagonalizability via $*$-congruence of Hermitian matrices.
An algorithm, that completely numerically solves the Hermitian SDC problem, 
is then proposed.
Some numerical tests are given as well in this section.
%Examples illustrating our methods are presented in Section \ref{sec:exam}. 
The conclusion and further discussion is devoted in the last section.

\section{Hermitian-SDC and SDS problems} \label{sec:sdc}

Recall that 
the SDS problem 
is that of finding a nonsingular matrix that simultaneous diagonalizes 
a collection of square complex matrices via similarity,
while the Hermitian SDC one is defined as earlier.

\subsection{SDC problem for commuting Hermitian matrices}

The following 
is presented in \cite[Theorem 4.1.6]{b398}
whose 
proof hides how to find a nonsingular matrix simultaneously diagonalizing the given matrices.
Our proof
is constructive and it may
lead to a procedure of finding such a nonsingular matrix.
It  follows that of Theorem 9 in \cite{jiang2016} for real symmetric matrices.

\begin{theo}\label{theo:H1}
	The matrices $I, C_1, \ldots, C_m \in \H^n,$ $m\geq 1,$ 
	are  SDC 
	if and only if they are commuting.
	Moreover, when this is the case,
	these are SDC by a unitary matrix, and the resulting diagonal matrices are all real.
\end{theo}
\begin{proof}
	If $I, C_1, \ldots, C_m$ are SDC then there is a nonsingular matrix $U \in \C^{n\times n}$ such that
	$U^*IU, U^*C_1 U, \ldots, U^* C_mU$ 
	are diagonal. 
	Note that 
	\begin{equation} \label{eq:theoH1p0}
	U^*IU =\diag(d_1, \ldots, d_n) \succ 0.
	\end{equation}
	Let $D= \diag(\frac{1}{\sqrt{d_1}}, \ldots, \frac{1}{\sqrt{d_m}})$
	and $V = UD.$
	Then 
	$V$ must be unitary and 
	%	Moreover, $V$ simultaneously diagonalizes (via congruence) the initial matrices because
	$$
	V^*C_iV = DU^* C_i U D \mbox{ is diagonal for every } i=1,\ldots,m.
	$$ 
	So  
	$
	(V^*C_iV)(V^*C_jV) = (V^*C_jV)(V^*C_jV), \enskip \forall i\neq j,
	$
	and hence $C_iC_j= C_jC_i$ for every $i\neq j.$
	It is worth mentioning that  each $V^*C_iV$ is real since it is Hermitian.

	We now prove the opposite direction by induction on $m.$
	
	The case $m=1,$ the proposition is clearly true since any Hermitian matrix can be diagonalized (certainly via congruence  as well as similarity) by a unitary matrix. 
	
	For $m\geq 2,$ we suppose the theorem holds true for $m-1.$
	
	We now consider an arbitrary collection of matrices $I, C_1, \ldots, C_m.$
	Let $P$ be a unitary matrix that diagonalizes $C_1:$ 
	$$
	P^*P= I, \enskip
	P^*C_1P = \diag(\alpha_1 I_{n_1}, \ldots, \alpha_k I_{n_k}),  
	%	\enskip
	%	\alpha_i \in \R,
	$$ 
	where 
	$\alpha_i$'s are distinctly real eigenvalues of $C_1.$
	Since $C_1,C_i$ commute for every $i=2, \ldots, m,$  
	so do $P^*C_1P$ and $P^*C_iP.$ 
	By Lemma \ref{lem:H1}, for every $i=2, \ldots,m,$ we have 
	$$P^* C_i P= \diag(C_{i1}, \ldots, C_{i k}),$$
	where every $C_{it}$ is Hermitian of order $n_t.$
	%	By diagonalizing $C_{ij}$ we obtain diagonal matrices
	%	$Q_{ij}^* C_{ij} Q_{ij}$
	%	for all $j=1, \ldots,k,$
	%	with
	%	$Q_{ij}\in \U^{n_j}.$
	%	Set $V_i = P \diag(Q_{i1}, \ldots, Q_{ik}),$ $i=2, \ldots,m.$
	
	Now, 
	for each $t=1, \ldots, k,$ since
	$C_{it} C_{j t} = C_{jt} C_{i t}$ for all $i,j=2, \ldots, m,$ 
	provided by $C_iC_j= C_jC_i,$
	the induction hypothesis leads to the fact that
	\begin{equation} \label{eq:theoH1} 
	I_{n_t}, C_{2t}, \ldots, C_{mt} 
	\end{equation}
	are SDC by a unitary matrix $Q_t.$
	Set
	$
	U = P\diag(Q_1, \ldots, Q_k).
	$
	Then
	\begin{align} \label{eq:theoH1p2}
	U^* C_1 U & = \diag(\alpha_1 I_{n_1}, \ldots, \alpha_k I_{n_k}), \\ \notag
	U^* C_i U & = \diag(Q_1^* C_{i1} Q_1, \ldots, Q_k^* C_{ik} Q_k), \enskip i=2, \ldots,m,
	\end{align}		
	are all diagonal.
\end{proof}

It is worth mentioning that
 the less number of multiple eigenvalues of the starting matrix $C_1$ in the proof of Theorem \ref{theo:H1},
 the less number of collections as in \eqref{eq:theoH1} must be solved.
 We 
keep into account this observation to the first step in the following.
%

%\noindent
%\rule[-1ex]{\linewidth}{1.0pt}
\begin{algo} \rm  \label{alg:withI}  
	%\noindent \textbf{\red{Algorithm 2}}. 
	Solving the SDC problem of commuting Hermitian matrices. 
%	\rule[1ex]{\linewidth}{0.5pt}

\noindent
\begin{tabular}{|p{\textwidth} | }
	\hline 
	\textsc{Input}: \hskip3mm
	Commuting matrices 
	$ C_1, \ldots, C_m \in \H^n.$
	
	\noindent 
	%\underline{
	\textsc{Output}:
	%} 
	A unitary %(orthogonal if the initial matrices are all real) 
	matrix $U$ that SDC the matrices $I, C_1, \ldots, C_m.$  
	%\rule{\linewidth}{1.5pt}
	%\vskip3mm
	\\ \hline
	
	\vspace{-3mm}
	\noindent
	%$\bullet$ \textsf{At Iteration} $i\geq 1:$	
	\begin{enumerate}[\quad \it Step 1:]
		\item
			Pick a starting matrix with the least number of multiple eigenvalues.
		\item 
		Find an eigenvalue decomposition of $C_1:$ 
		$C_1 = P^* \diag(\lambda_1 I_{n_1}, \ldots, \lambda_k I_{n_k}) P,$
		$n_1+ \ldots + n_k =n,$
		$\R\ni \lambda_i$'s are distinct and
		$P^* P = I.$
		
		\item
		Computing diagonal blocks of $P^*C_iP,$ $i\geq 2:$
		$$
		P^*C_iP = \diag(C_{i1}, \ldots, C_{ik}), \enskip 
		C_{it} \in \H^{n_i}, \ \forall t=1, \ldots, k,
		$$
		where
		$C_{2t }, \ldots, C_{mt}$ pairwise commute for each $t=1, \ldots, k.$
		
		\item 
		For each $t=1, \ldots, k,$ 
		simultaneously diagonalizing the collection of matrices $I_{n_t}, C_{2t }, \ldots, C_{mt}$ 
		by a unitary matrix $Q_t.$
		
		\item
		Define	$U = P\diag(Q_1, \ldots, Q_k)$ and $V^* C_iV.$ 
	\end{enumerate}
	\vspace{-5mm}
\\ \hline
\end{tabular}
\end{algo}

\subsection{An equivalent condition via the SDS}
Using Theorem \ref{theo:H1}
we comprehensively describe the SDC property of a collection of Hermitian matrices as follows.
It is worth mentioning that the parameter $\lambda$ appearing in the following theorem is always real even if $\F$ is the field of real or complex numbers.
\begin{theo}\label{theo:maxrk} 
	Let $C_1, \ldots, C_m \in \F^{n\times n} \setminus\{0\}$ be Hermitian with 
	$\dim_{\F} \left(\bigcap_{t=1}^m \ker C_t \right) =q $ (always $q<n$).
	\begin{enumerate}[\rm 1)]
		\item
		When $q=0,$ %there are two following cases:
		\begin{enumerate}
			\item 
			If $\det \bC(\lambda) =0$ for all $\lambda \in \R^m$ (only real $m$-tuples $\lambda$) then 
			$C_1, \ldots, C_m$ are not SDC (on $\F$);
			
			\item
			Otherwise, $\bC(\lambda)$ is nonsingular for some $\lambda \in \R^m.$ 	
			%			Suppose $\bC= \bC(\lambda)$ is a such nonsingular linear combination. 
			The matrices
			$C_1, \ldots, C_m$ are $\F$-SDC 
			if and only if
			they
			$\bC(\lambda)^{-1} C_1, \ldots, \bC(\lambda)^{-1}C_m$ pairwise commute and 
			every $\bC(\lambda)^{-1}C_i,$ $i=1, \ldots,m,$ 
			is similar to a real diagonal matrix.
		\end{enumerate}	
		
		\item 
		If $q>0$ then there exists a nonsingular matrix $P$ such that 
		\begin{equation} \label{eq:maxrk1}
		P^* C_iP = 
		\begin{bmatrix}
		0_q & 0 \\ 0 & \hat{C}_i
		\end{bmatrix},
		\enskip \forall i=1, \ldots, m,
		\end{equation}
		where  $0_q$ is the $q\times q$ zero matrix and $\hat{C}_i \in \H^{n-q}$ with
		$\bigcap_{t=1}^m \ker \hat{C}_t = 0.$
		
		Moreover, $C_1, \ldots, C_m$ are $\F$-SDC if and only if $\hat{C}_1, \ldots, \hat{C}_m$ are SDC.		
	\end{enumerate} 
\end{theo}
\begin{proof}
	1) 	Suppose $\dim_{\F} \bigcap_{t=1}^m \ker C_t=0.$
	
	For the part (a), the fact that
 $C_1, \ldots, C_m$ are SDC by a nonsingular matrix $P\in \F^{n\times n}$ 
 implies
	$$
	C_i = P^* D_i P, \enskip
	D_i = \diag(\alpha_{i1}, \ldots, \alpha_{in}), \enskip \forall i=1, \ldots,m,
	$$ 
	where we note $D_i$ is real since $C_i = C_i^*.$
	This follows that the real polynomial (with real variable $\lambda$)
	\begin{equation} \label{eq:C-lambd}
	\det \bC(\lambda) = |\det(P)|^2 \prod_{j=1}^n (\sum_{i=1}^{m}  \alpha_{ij} \lambda_i)
	\end{equation}
	is identically zero because of the hypothesis.
	Since $\R[\lambda_1, \ldots, \lambda_m]$ is an integral domain, 
	there exists 
	a factor identically zero, say,
	$(\alpha_{1j}, \ldots, \alpha_{mj}) = 0$ for some
	$j =1, \ldots, n.$ 
	Pick a vector $x$ satisfying $P x = \mathbf{e}_j,$
	the $j$-th unit vector in $\F^n,$
	one obtains
	$$
	C_i x = P^* D_i P x = P^* D_i \mathbf{e}_j =0.
	$$
	This means 
	$0 \neq x \in \bigcap_{i=1}^m \ker C_i,$
	contradicting to the hypothesis.
	This proves the part (a).
	
	We now prove the part (b). Suppose $\bC(\lambda) $ is nonsingular for some $\lambda \in \R^m.$
	If $C_1, \ldots, C_m$ are SDC
	then
	there exists a nonsingular matrix $P\in \F^{n\times n}$ such that
	$P^* C_i P \in \F^{n\times n}$ are all diagonal and real, and so is 
	$P^* \bC(\lambda) P = [P^* \bC(\lambda) P]^*$ since $\lambda \in \R^m.$  
	Then 
	\begin{equation} \label{eq:maxrk2}
	P^{-1} \bC(\lambda)^{-1} C_i P = [P^* \bC(\lambda) P]^{-1} (P^*C_i P)
	\end{equation}
	is diagonal for every $i=1, \ldots, m.$
	
	Conversely, take a nonsingular matrix $P\in \F^{n\times n}$ such that
	\begin{equation} \label{eq:maxrk3}
	P^{-1} \bC(\lambda)^{-1} C_1 P:= D_1 \in \R^{n\times n}
	\end{equation}
	is diagonal.
	Up to a rearrangement of diagonal elements, we can assume
	$$ 
	D_1 = \diag(\alpha_{1} I_{n_1}, \ldots, \alpha_{k} I_{n_k}), \enskip 
	n_1 + \ldots+n_k =n, \enskip \alpha_i \neq \alpha_j\ \forall i\neq j.
	$$	
	For every $i=2, \ldots, m,$ since
	$D_1$ and $P^{-1} \bC(\lambda)^{-1} C_i P$ 
	commute.
	Lemma \ref{lem:H1} implies 
	$$ 
	P^{-1} \bC(\lambda)^{-1} C_i P = \diag(C_{i1}, \ldots, C_{ik}), \enskip
	C_{it}\in \F^{n_t \times n_t}, \enskip
	\forall t =1, \ldots, k.
	$$
	The proof immediately complete after two following claims are proven.
	
\noindent
\textit{Claim 1}. \textit{For each $t=1, \ldots, k,$  the collection 		
	$
	C_{2t}, \ldots, C_{mt},
	$ 
	is SDS by a nonsingular matrix $Q_t.$}	
	Indeed,
	since $P^{-1} \bC(\lambda)^{-1} C_i P$ 
	is diagonalizable via similarity,
	so is $C_{it}\in \F^{n_t \times n_t}$ for every $t=1, \ldots,k,$
	provided by Lemma \ref{lemm:HB1}(iii).
	Moreover, the pairwise commutativity of
	$\bC(\lambda)^{-1} C_2,$
	$ \ldots, $
	$\bC(\lambda)^{-1} C_m$
	implies that of 
	$P^{-1} \bC(\lambda)^{-1} C_2 P,$
	$ \ldots, $
	$P^{-1} \bC(\lambda)^{-1} C_mP;$
	and hence that of
	$$
	C_{2t}, \ldots, C_{mt},
	$$ 
	for each $t=1, \ldots, k.$
	By Lemma \ref{lemm:HB1}, 
	the later matrices are thus SDS by an invertible matrix $Q_t,$ i.e., 
	$$
	Q_t^{-1}C_{2t} Q_t =: D_{2t}, \enskip \ldots, \enskip Q_t^{-1}C_{mt} Q_t =: D_{mt}
	$$
	are all diagonal.

\noindent
\textit{Claim 2}. 
\textit{There exists a nonsingular matrix $U$ such that $U^*C_1U, \ldots, U^*C_mU$ pairwise commute. These later matrices are then $\F$-SDC by Theorem \ref{theo:H1}}. Indeed, 
	let
		$$U = P \diag(Q_1, \ldots, Q_k),$$
		we then have
		\begin{align*}
		U^{-1} \bC(\lambda)^{-1} C_1 U   = \diag(\alpha_{1} I_{n_1}, \ldots, \alpha_{k} I_{n_k})  &= D_1   \\
		U^{-1} \bC(\lambda)^{-1} C_i U = \diag(D_{i1}, \ldots, D_{ik}) &  = D_i, \enskip
		i=2, \ldots, m,
		\end{align*}
		Note that 
		$D_i$ is real for every $i=1, \ldots,m$ because of the hypothesis.
		Since $C_i$ is Hermitian, 
		the equality
		\eqref{eq:maxrk2} implies
		$$
		[U^*\bC(\lambda) U] D_i %= (U^*\bC U)(U^{-1} \bC^{-1} C_i U) 
		= U^*C_i U = (U^*C_i U)^* = D_i [U^* \bC(\lambda) U], \enskip 
		\forall i=1, \ldots,m.
		$$
		Then
		\begin{align*}
		(U^*C_i U) (U^* C_jU) 
		%&= (P^* \bC P) (P^{-1} \bC^{-1} C_i P ) (P^* \bC P) (P^{-1} \bC^{-1} C_j P) \\
		&= [U^* \bC(\lambda) U  (D_i D_j) U^* \bC(\lambda) U ]
		= [U^* \bC(\lambda) U (D_jD_i)  U^* \bC(\lambda)_j U  ]\\	
		&= (U^*C_jU) (U^* C_iU) ,
		\end{align*}
		for every $i\neq j.$
		Theorem \ref{theo:H1} allows us to finish this part.

	2) 	
	Pick an orthonormal basis 
	$u_1, \ldots, u_q, u_{q+1}, \ldots, u_n$ 
	of  the $\F$-unitary vector space $\F^n \equiv \F^{n\times 1}$ 
	such that 
	$u_1, \ldots, u_q$
	is an orthonormal basis of 	 
	$\bigcap_{t=1}^m \ker C_t.$
	The matrix $P$ whose the columns are $u_j$'s will satisfy the conclusion that 
	$P^*C_iP = \diag(0_q, \hat{C}_i)$ and $\bigcap_{t=1}^m \ker\hat{C}_t = 0.$
	
%	The last claim is trivial since the fact that

Finally, we already know that
$C_1 , \ldots, C_m$ are SDC if and only if so are
$P^*C_1P,$ 
$ \ldots,$ 
$P^*C_mP,$
ans so are 
$\hat{C}_1, \ldots, \hat{C}_m$
(see Lemma \ref{lem:appdx1} in  \ref{app:1}). 
\end{proof}

\begin{algo} \rm \label{alg:detect} 
	%\noindent \textbf{\red{Algorithm 1}}. 
	Detecting whether a collection of Hermitian matrices is SDC or not.

\noindent
\begin{tabular}{|p{.98\textwidth} | }
	\hline 
	\textsc{Input}: \hskip3mm
	Matrices 
	$C_1, \ldots, C_m \in \H^n$
	(not necessary pairwise commute).

	\noindent 
	%\underline{
	\textsc{Output}:
	%} 
	Conclude whether $C_1, \ldots, C_m $ are SDC or not.
	%\rule{\linewidth}{1.5pt}
	%\vskip3mm
\\ \hline		
\vspace{.1mm}
	Compute a singular value decomposition $C  = U \Sigma V^* $
	of $C = [C_1^* \enskip \ldots \enskip C_m^*]^*,$
	$\Sigma = \diag(\sigma_1, \ldots, \sigma_{n-q}, 0, \ldots,0),$\
	$\sigma_1 \geq \ldots \geq \sigma_{n-q} >0,$
	$0\leq q \leq n-1.$
	Then $\dim_{\F} \left(\bigcap_{t=1}^m \ker C_t \right) =q.$ 
\begin{enumerate}
\item [IF \ \ \ ]
		$q=0:$ 
		\begin{enumerate}[\textit{Step} 1:]
			\item
			If $\det \bC(\lambda) =0$ for all $\lambda \in \R^m$ then 
			$C_1, \ldots, C_m$ are not SDC.
			
			Else, go to \textit{Step} 2.
			
			\item
			Find a $\underline{\lambda} \in \R^m$ such that $\bC:= \bC(\underline{\lambda})$ is nonsingular. 
			\begin{enumerate}[(a)]
				\item
				If $\bC^{-1}C_i$ is not similar to a diagonally real matrix for some $i=1, \ldots, m,$ then conclude the given matrices are not SDC.
				
				Else, go to (b).
				\item
				If
				$\bC^{-1}C_1, \ldots, \bC^{-1}C_m$ do not pairwise commute, which is equivalent to that $C_i\bC^{-1}C_j$ is not Hermitian for some pair $i\neq j,$ 
				then conclude the given matrices are not SDC. 
				
				Else, conclude the given matrices are SDC. 
			\end{enumerate} 
		\end{enumerate}
\item[ELSE]
		$(q>0):$ 
		\begin{enumerate}[\hskip-5mm\textit{Step} 1:]
			\item[\textit{Step} 3:]
			For $C  = U \Sigma V^*$ determined at the beginning, 
			the $q$ last columns of $V,$ say $\vec{v}_{1}, \ldots, \vec{v}_q,$ span $\ker C = \bigcap_{t=1}^m \ker C_i.$
			Pick a nonsingular matrix $P\in \F^{n\times n}$ whose $q$ first columns are $\vec{v}_1, \ldots, \vec{v}_q.$
			Then $P$ satisfies \eqref{eq:maxrk1}.
			
			\item[\textit{Step} 4:]
			Apply \textit{Case} 1 to the resulting matrices $\hat{C}_1, \ldots, \hat{C}_m \in \H^{n-q}.$
		\end{enumerate}
	\end{enumerate}
\vspace{-8mm}
\\ \hline
\end{tabular}
\end{algo}

Note that (see also Lemma \ref{lem:rk}  in Appendix \ref{app:1})
$$
\det \bC(\lambda) =0 \enskip \forall \lambda \in \R^m
\Longleftrightarrow
\max\{ \rk \bC(\lambda)|\ \lambda \in \R^m\} <n,
$$
and the later condition is easier  checked in practice.
We hence prefer Algorithm \ref{alg:Schm} below for checking Step 1 of Algorithm \ref{alg:detect}.
Since the set of $n\times n$ Hermitian matrices is the Hermitian part of the $*$-algebra of $n\times n$ complex matrices.
Thanks to 
Schm\"{u}dgen's procedure for diagonalizing a Hermitian matrix over a commutative $*$-algebra
\cite{Schmudgen09},
 we  provide a procedure for finding such a maximum rank of $\bC(\lambda).$ 
This technique may be possible to apply to some other types of matrices, for example, complex symmetric ones \cite{Bustamante2020}.

\subsection{Finding maximum rank of  a Hermitian pencil}
\label{sec:schm}

Schm\"{u}dgen's procedure %for diagonalizing a matrix over a commutative ring
 \cite{Schmudgen09}
is summarized as follows:
for $F \in \H^n$ partitioned as
$$
F = 
\begin{bmatrix}
\alpha & \beta \\ \beta^* & \hat{F}
\end{bmatrix} ,
\enskip
(\alpha \in \R),
$$ 
we then have the following relations
\begin{equation}\label{eq:SchmProc1}
X_+ X_- = X_- X_+ = \alpha^2 I, \quad
\alpha^4 F = X_+ \widetilde{F} X_+^*, \quad
\widetilde{F}= X_- F X_-^*,
\end{equation}
where
\begin{equation}\label{eq:SchmProc2}
X_\pm = 
\begin{bmatrix}
\alpha &0 \\ \pm \beta^* & \alpha I
\end{bmatrix}, \quad
\widetilde{F} =
\begin{bmatrix}
\alpha^3 &0 \\ 0  & \alpha (\alpha \hat{F} - \beta^*\beta )
\end{bmatrix}
:= 
\begin{bmatrix}
\alpha^3 &0 \\ 0  & F_1
\end{bmatrix}
\in  \H^n.
\end{equation}

We now apply the above to the pencil
$F = \bC(\lambda) = \lambda_1 C_1 +\ldots+\lambda_m C_m,$
where $C_i \in \H^n,$ 
 $\lambda\in \R^m.$
 In our situation of Hermitian matrices,
we have have the following together with
a constructive proof that leads to a procedure for determining a maximum rank linear combination.

\begin{lemm}\label{lem:Schmud}
	Let $\bC= \bC(\lambda)\in \F[\lambda]^{n\times n} ,$ $\lambda \in \R^m,$ be a pencil
	satisfying $\bC^* =\bC.$ 
	Then there exist polynomial matrices 
	$\bX_+, \bX_- \in \F[\lambda]^{n\times n}$ 
	and polynomials $b, d_j \in \F[\lambda],$ $j=1, \ldots, n,$  %$r\leq m,$
	such that
	\begin{subequations}
	\begin{align}
	\bX_+ \bX_- 		&= \bX_- \bX_+   = b^2 I_n,  \label{eq:SchmudDec1}\\
	b^4 \bC 		&= \bX_+ \diag(d_1, \ldots, d_n) \bX_+^ *,  \label{eq:SchmudDec2} \\
	\bX_- \bC \bX_-^* & = \diag(d_1, \ldots, d_n). \label{eq:SchmudDec3}
	\end{align}
	\end{subequations}
	%	Note that there may be some $d_i$ is identically zero for every
	%	$
	%	i\geq r= \max\{\rk_{\F} \bC(\lambda)|\ \lambda \in \F^m\}.
	%	$ 
\end{lemm}
\begin{proof}
	The lemma is constructively proved as follows.
	The procedure in this proof will stop when $\bC_k$ is diagonal.
	It is shown in, eg.,  \cite{Schmudgen09} and \cite{Cimpric2012}, that if the $(1,1)$st entry of $\bC$ is zero then one can find a nonsingular matrix $\bP$ for that of $\bP\bC \bP^*$ is nonzero. 
	In logically similar way,  we can assume  
	every matrix $\bC_k$ which is applied Schm\"{u}dgen's procedure at every step below
	has nonzero $(1,1)$st entry. 
	
	At the first step, 
	partitioning 
	$\bC = \bC^*$ 
	as
\begin{equation}\label{eq:SchmStep0}
	\bC = 
\begin{bmatrix}
\alpha & \beta \\ \beta^* & \hat{\bC}_1
\end{bmatrix}, \enskip
\hat{\bC}_1^* = \hat{\bC}_1 \in \F^{(n-1) \times (n-1)}, \enskip 
0 \neq \alpha \in \R[\lambda] .
\end{equation}
Assigning 
$\alpha_1 = \alpha,$ 
$\beta_1 = \beta,$
$\bC_1=  \alpha_1(\alpha_1 \hat{\bC}_1 - \beta_1^* \beta_1) \in \H^{n-1}$ 
and
	$$
	\bX_{1\pm} := \bY_{1\pm}(\lambda) =
	\begin{bmatrix}
	\alpha_1 & 0 \\ \pm \beta_1^* & \alpha_1 I_{n-1}
	\end{bmatrix}
	$$
	as in \eqref{eq:SchmProc2},
	then
$$
		\begin{array}{llll}
		\bX_{1+} \bX_{1-} 		&= \bX_{1-}\bX_{1+}  = \alpha_1^2 I_n, \\
		X_{1-} \bC X_{1-} ^{*} 
		& = 
		\begin{bmatrix} 
		\alpha_1^3 & 0\\ 0 & \bC_1 \end{bmatrix} := \tilde{\bC}_1 ,
		&
		\alpha_1^4 \bC 		
		&= 
		\bX_{1+}
		\tilde{\bC}_1
		%	\begin{bmatrix} 
		%	\alpha_1^3 & 0\\ 0 &  \bC_1  
		%	\end{bmatrix} 
		\bX_{1+}^ *.
		\end{array}
$$
	If  
	$\bC_1$
	is diagonal then one stops.
	Otherwise,
	let us partition
	$
	{\bC}_1 = 
	\begin{bmatrix}
	\alpha_2 & \beta_2 \\ 
	\beta_2^* & \hat{\bC}_2
	\end{bmatrix}
	$
	and
	 continue  applying Schm\"{u}dgen's procedure to $\bC_1$ in the second step
	$$
	{\bY}_{2\pm} =
	\begin{bmatrix}
	\alpha_2 & 0 \\ \pm \beta_2^* & \alpha_2 I_{n-1}
	\end{bmatrix}, \enskip
	\bY_{2-} \bC_1\bY_{2-}^* = 
	\begin{bmatrix} 
	\alpha_2^3 & 0\\ 
	0 & \bC_2
	\end{bmatrix} , \enskip
%	\bC_2 = \alpha_2(\alpha_2 \tilde{\tilde{\bC}} - \tilde{\beta}^* \tilde{\beta}) , \enskip
\bC_2 \in \H^{n-2},  
	$$
	where $\alpha_2 = \bC_2(1,1),$ the $(1,1)$st entry of $\bC_2 = \alpha_2(\alpha_2 \hat{\bC}_2 - \beta_2^*\beta_2).$
	The updated matrices
		$$
	\bX_{2-} = 
	\begin{bmatrix}
	\alpha_2 & 0 \\
	0  & \bY_{2-} 
	\end{bmatrix}
	\bX_{1-}, \enskip
	\bX_{2+} = 
	\bX_{1+}
	\begin{bmatrix}
	\alpha_2 & 0 \\
	0  & \bY_{2+}
	\end{bmatrix}
	$$
	and 
	\begin{align*}
	\bX_{2-} \bC \bX_{2-}^*
	= 
	\begin{bmatrix}
	\alpha_1^3 \alpha_2^2 		& 0 							& 0 \\
	0								& \alpha_2^3 		& 0 \\
	0								&0								& \bC_2
	\end{bmatrix} 
%	& := 
=
	\begin{bmatrix}
\alpha_2^2	\diag(\alpha_1^3,  \alpha_2)     & 0\\
	0			& \bC_2
	\end{bmatrix}
	= \tilde{\bC}_2
	\end{align*}
then satisfy the relations \eqref{eq:SchmProc1}:
	$
	\bX_{2-}\bX_{2+} = \bX_{2+}\bX_{2-} = \alpha_1^2 \alpha_2^2 I=b^2 I.
	$
	The second step completes.
	
	Suppose now we have at the $(k-1)$th step that
	$$
	\bX_{(k-1)-} \bC \bX_{(k-1)-}^*
	= 
	\begin{bmatrix}
	\diag(d_1, \ldots, d_{k-1})  		  							& 0 \\
	0								 							& \bC_{k-1}
	\end{bmatrix} := \tilde{\bC}_{k-1},	
	$$
	where 
	$	\bC_k =  \bC_k^* \in \F[\lambda]^{(n-k+1) \times (n-k+1)},$
	and
	 $d_1, \ldots, d_{k-1}$ are all not identically zero.
	If $ \bC_{k-1}$ is not diagonal (and suppose that its $(1,1)$st entry is nonzero) 
then partition $\bC_{k-1}$
and
compute as follows:
	\begin{align} \label{eq:SchmStepk}
		{\bC}_{k-1} &= 
	\begin{bmatrix}
	\alpha_k & \beta_k \\ 
	\beta_k^* & \hat{\bC}_k
	\end{bmatrix}, \enskip 
	\alpha_{k}  = \bC_{k-1}(1,1), \enskip
		\bC_k  = \alpha_{k-1}(\alpha_{k-1} \hat{\bC}_k - \beta_{k-1}^*\beta_{k-1}), \enskip \notag 
  \\ 
	\bX _{k+} 	& = 				
	\bX_{(k-1)+}  
	\left[
	\begin{array}{cccc}
	\alpha_kI_{k-1}&		&			&0\\
	0			&		&			&{\bY}_{k+}
	\end{array}
	\right],\enskip 
	\bX_{k-}  = 
	\left[
	\begin{array}{cccc}
	\alpha_k	I_{k-1}&		&			&0\\
	0			&		&			&{\bY}_{k-}
	\end{array}
	\right]
	\bX_{(k-1)-},			\notag  
	\\
	\tilde{\bC}_k  &= 
	\left[
	\begin{array}{clll}
    \diag\left(d_1, \ldots, d_{k-1}, d_k \right)	&		&			&0\\
	0			&		&			& 					
	\bC_k
	\end{array}
	\right]
	= \bX_{k-} \bC \bX_{k-}^*, \notag \\
		b  	&= \prod_{t=1}^{k} \alpha_t , 	
	\end{align}
	where
%	Note that at $k$th step, the diagonal elements of $\bD_k$ are determined as 
	\begin{equation} \label{eq:dj}
	d_k = \alpha_k^3, \enskip 
	d_j = \alpha_j^3 \prod_{t=j+1}^k \alpha_t^2, \enskip 
	j=1, \ldots, k-1.
	\end{equation}
	The procedure will stop if $\bC_k$ is diagonal,
%	One can check that
%	the relations \eqref{eq:SchmProc1} and \eqref{eq:SchmProc2} immediately satisfy.
%	We then pick 
and one picks
	$\bX_{\pm} = \bX_{k\pm}$ 
	that diagonalizes $\bC$ as in \eqref{eq:SchmudDec3}.
\end{proof}

The following
allows us to determine a maximum rank linear combination.

\begin{coro}\label{coro:Schm}
	With notations as in Lemma \ref{lem:Schmud}, and 
	suppose the procedure stops at step $k.$
	That is,
$\bC_k$ in \eqref{eq:SchmStepk} is diagonal but so are not
	$\bC_t$  for all $t=1, \ldots,k-1.$

\begin{enumerate} [ \rm i)]
\item 
	Assume further that the diagonal matrix $\tilde{\bC}_k$ has the form as the right-hand side of \eqref{eq:SchmudDec3}.
	For $\lambda \in \R^m,$
	if $d_j(\lambda) \neq 0$ for some $j =1, \ldots,n$
	then 
	$d_t(\lambda) \neq 0$ for every $t=1, \ldots, j.$
	
\item
	The pencil $\bC(\lambda)$ has maximum rank $r$ if and only if
 	$d_j $ is identically zero for all $j= r+1, \ldots, n,$
	and   
	there exists $\hat{\lambda} \in \R^m$ such that
	$
	b(\hat{\lambda}) \prod_{t=1}^r d_t(\hat{\lambda}) \neq 0.
	$
\end{enumerate}	
\end{coro}
\begin{proof}
i) 
	As shown in \eqref{eq:SchmStepk},
	the fact
	$\bC_{t+1} = \alpha_t(\alpha_t \hat{\bC}_{t} -\beta_t\beta_t)$
	and $\alpha_{t+1} = \bC_t(1,1)$
	for every $t=1, \ldots, k-1,$
	implies that $\alpha_t$ divides $\alpha_{t+1}.$
	In particular,
	$\alpha_k$ is divisible by $\alpha_t$ for every $t=1, \ldots,k.$
	Moreover,
	$\alpha_k $ divides $d_s$
		for every $s = k+1, \ldots,n,$
		provided  by the first-row of \eqref{eq:SchmStepk}. 
		The claim is hence immediately followed.

ii)
	If $d_j \not \equiv 0$ for some $j >r$ then,
	by the previous part,
	$d_j(\hat{\lambda}) \neq 0,$ and so are $d_t(\hat{\lambda}),$ $t\leq j,$ for some 
	$\hat{\lambda} \in \R^m.$
	This means $\bC(\hat{\lambda})$ has rank $j> r,$ a contradiction.
	Thus $d_j \equiv 0$ for every $j >r.$
	
	Now, if $k>r$ then $d_k =\alpha_k \equiv 0.$
	This is impossible	because the procedure proceeds only when $\bC_t(1,1) \neq 0$ at each step $t.$
	This yields $k\leq r.$
	This certainly implies $b(\hat{\lambda}) \neq 0$ since 
	$b = \alpha_1 \cdots \alpha_k.$
	
	The opposite direction is obvious.
\end{proof}

\begin{rema}\rm 
	The proof of Lemma \ref{lem:Schmud} provides a comprehensive update according to Schm\"{u}dgen's procedure.
	But in our situation,	
	only the diagonal elements of $\tilde{\bC}_k,$
	from which
	one can determine a $\hat{\lambda} \in \R^m$ at the end,
	are needed.
	So, in the computations of \eqref{eq:SchmStepk},
	one does not need to update

	From \eqref{eq:SchmStepk},
	it suffices to find
	Corollary \ref{coro:Schm}
	only \eqref{eq:SchmudDec3} is needed.
	However, the last update formula in \eqref{eq:SchmStepk} shows that the diagonal elements of $\bD$ are not simply updated.
	Indeed, at step $k,$ we can update as follows:
	\begin{align} \label{eq:SimUpdate}
	\alpha_k &= \bC_{k-1}(1,1), \quad
	b  			= b \alpha_k ,  \quad
	\bX_{k-}  = 
	\left[
	\begin{array}{cccc}
	I_{k-1}&		&			&0\\
	%	&\ddots	&			&\\
	0			&		&			&\tilde{\bX}_{-}
	\end{array}
	\right]
	\bX_{(k-1)-},			 \\
	\tilde{\bC}_k  &= 
	\left[
	\begin{array}{clll}
	\diag\left(\alpha_1^3, \ldots, \alpha_{k-1}^3, \alpha_k^3 \right)	&		&			&0\\
	0			&		&			& 					
	\bC_{k+1}
	\end{array}
	\right]. \nonumber
	%			 			d_{i}^{(i-1)} = \alpha_i^3.
	\end{align}
	Moreover, 
	if $\bC_{k+1}$ has order greater than 1 then 
	$\alpha_t$ divides $\alpha_{t+1}$ for all $t=1, \ldots, k.$
\end{rema}

\begin{algo}   \label{alg:Schm}  
	Schm\"{u}dgen-like algorithm determining maximum rank of a pencil.
	%	\rule[1ex]{\linewidth}{0.5pt}
	
	\noindent
	\begin{tabular}{|p{\textwidth} | }
		\hline 
		\textsc{Input}: \hskip3mm
		Hermitian matrices
		$ C_1, \ldots, C_m \in \H^n.$
		
		\noindent 
		%\underline{
		\textsc{Output}:
		%} 
		A real
		$m$-tuple $\hat{\lambda} \in \R^m$ that maximizes the rank of the pencil $\bC(\lambda).$
		\\ \hline
		\vspace{-4mm}
		\noindent
		%$\bullet$ \textsf{At Iteration} $i\geq 1:$	
	%	
		\begin{enumerate}[\quad \it Step 1:]
			\item
			Partitioning $\bC(\lambda)$ as in \eqref{eq:SchmStep0}.
			\item 
 			At the iteration $k\geq 1:$ 
 			
 			+ If $\bC_{k}$ is not diagonal then 
 			do the computations as in \eqref{eq:SchmStepk} and
 			go to the iteration $k.$
 			
 			+ Else ($\bC_k$ is diagonal),
 			take into account 
 			$\tilde{\bC}_k$ is diagonal whose first $k$ elements $d_1, \ldots, d_k$ are not identically zero and go to Step 3.
 			
 		\item
 			Pick $\hat{\lambda} \in \R^m$ such that 
 			$d_k(\hat{\lambda}) \neq 0$ and
 			return $\hat{\lambda}.$
		\end{enumerate}
	\vspace{-8mm}
		\\ 				
		\hline
	\end{tabular}
\end{algo}

\subsection{Examples} \label{sec:exam1}

\begin{exam} \label{ex:2}
	Let
	$$
	C_1=\left(
	\begin{matrix}1 &3& -2\\3&16&-10\\-2&-10&6\end{matrix}
	\right), 
	C_2=\left(\begin{matrix}0 &0& 0\\0&-3&2\\0&2&-1\end{matrix}\right),
	C_3=\left(\begin{matrix}-1 &-3& 2\\-3&-5&4\\ 2&4&-3\end{matrix}\right).
	$$ 
	One can check that
	$\ker C_1 \cap \ker C_2 \cap C_3 = \{0\}$ since $\rk C_1 =3.$		

	We consider two matrices
	$$
	M_2 = C_1^{-1} C_2  =
	\begin{bmatrix}
	0  &  1  &   -1 \\
	0   &-1  &  0\\
	0   & -1  &   -\frac{1}{2}
	\end{bmatrix}, \quad
	M_3 = C_1^{-1} C_3  =
	\begin{bmatrix}
	-1 	&   -5 	&    3 \\
	0 		&    0	&  0\\
	0 		&   -1 	&  \frac{1}{2}
	\end{bmatrix},
	$$
	which has distinct eigenvalues $-1, \frac{1}{2}, 0$ and
	$-1, - \frac{1}{2}, 0,$ respectively, and hence they are diagonalizable via similarity. 
	Moreover, $M_2$ and $M_3$ commute. Theorem \ref{theo:maxrk} 
	yields that $C_1, C_2, C_3$ are SDC.
	\hfill$\diamond$
\end{exam}

We now consider other examples in which all given matrices are singular.

\begin{exam}\label{ex:3}
	The matrices
	\begin{align*}
	C_1 &= 
	\begin{pmatrix}
	1    & 3   & -1  \\  3  &  6   &  0   \\  -1    & 0   & -2
	\end{pmatrix}, \enskip
	C_2 = 
	\begin{pmatrix}
	0 & 0 & 0 \\ 0 & -3 & 2 \\ 0 &2 & -1
	\end{pmatrix}	, \enskip
	C_3 = 
	\begin{pmatrix}
	-1 &  -3 &  2 \\  -3 &  -5 &  4 \\   2 &  4 & -3
	\end{pmatrix}	.
	\end{align*}
	are all singular since
	$\rk(C_1) = \rk(C_2) = \rk(C3) =2.$
	
On the other hand,
	$\dim (\ker C_1 \cap \ker C_2 \cap C_3) =0$
	since $\rk[C_1 \enskip C_2 \enskip C3]^T=3.$
	
	Consider the linear combination
	$$
	\bC = x C_1 + y C_2+ z C_3 =
	\begin{bmatrix}
	x - z			&  3 x - 3 z					&       2 z -  x \\
	3 x - 3 z		& 6 x - 3 y - 5 z 		&   2 y + 4 z \\
	2 z -x		&    	 2 y + 4 z 			& - 2 x - y - 3 z
	\end{bmatrix}.
	$$	
	Applying Scm\"{u}dgen's procedure we have
	$$
	\bX_- \bC \bX_-^* 
	= 
	\begin{bmatrix}
	(x-z)^3 						&0 					 \\
	0									 	&	\bC_1				  \\
	\end{bmatrix}, \enskip
	\bX_-   =
	\begin{bmatrix}
	x - z			&     0			&     0\\
	3 z - 3 x		& x - z			&     0\\
	x-2 z 			&     0			& x - z
	\end{bmatrix} ,
	$$
	where
	\begin{align*}
	\bA &= (x-z)
	\begin{bmatrix}
	3 y - 6 x + 5 z- 9(  x -   z)^2  		& 	(3x   +  2 y -2 z  ) (x - z)  \\
	(3  x   +  2 y -2 z ) (x - z)      &      -( x - 2 z)^2 -(x - z) (2 x + y + 3 z)
	\end{bmatrix} 
	\\
	&:= 
	\begin{bmatrix}
	\alpha		& \beta \\
	\beta	    &   \gamma
	\end{bmatrix}.
	\end{align*}
	Let
	$$
	\bX_{i-} = 
	\begin{bmatrix}
	1 & 0 & 0\\
	0 & \alpha & 0\\
	0 & -\beta & \alpha
	\end{bmatrix}.
	$$
	We then have
	\begin{align*}
	\bX_{i-} (\bX_- \bC \bX_-^*) \bX_{i-}^* &= 
	\begin{bmatrix}
	(x-z)^3 & 0  			& 0\\
	0	& \alpha^3 		& 0\\
	0	& 0				& \alpha(\alpha\gamma - \beta^2)
	\end{bmatrix},
	\end{align*}
	where 
	\begin{align*}
	\alpha &= (x-z)[3y-6x+5z -9(x-z)^2],\\
	\beta  &= (3x+2y-2z)(x-z)^2,\\
	\gamma &= -(x-z)(x-2z)^2-(x-z)^2(2x+y+3z)
	\end{align*}
	If we pick $(x,y,z)=(2,0,3)$ then 
	$\alpha = 6,$
	$\beta = 0,$
	$\gamma=3,$
	and 
	$\alpha(\alpha \gamma -\beta) = 108\neq 0.$
	Then 
	$$
	\bX_- =
	\begin{bmatrix}
	-1  & 0 &  0\\
	3   &-1 &  0\\
	-4  & 0 & -1 
	\end{bmatrix}, \enskip
	\bC =2 C_1  + 3C_3 =
	\begin{bmatrix}
	-1   &-3  &  4\\
	-3   &-3  & 12\\
	4    &12  &-13
	\end{bmatrix}, \enskip
	\rk \bC =3.
	$$ 
	Note that for $\beta =0$ then we do not need to compute $\bX_{i-}$ since $\bA$ is diagonal.
	In this case,
	$\mathbf(C)^{-1}C_1, \mathbf(C)^{-1}C_2, \mathbf(C)^{-1}C_3$
	all have real eigenvalue 
	but
	$(\bC^{-1} C_1 ) (\bC^{-1} C_2 ) \neq (\bC^{-1} C_2 )(\bC^{-1} C_1 ),$
	the initial matrices are not SDC 
	by Theorem \ref{theo:maxrk}.
	\hfill$\diamond$
\end{exam}

\begin{exam}\label{ex:4}
	The matrices
	$$
	C_1=
	\begin{bmatrix}
	-1  & -4  &  4\\
	-4  & -16 &  16\\
	4   &  16 & -16
	\end{bmatrix}, 
	C_2= \begin{bmatrix}0 &0& 0\\0&-1&2\\0&2&-4\end{bmatrix} ,
	C_3= 
	\begin{bmatrix}
	-1 & -3  & 2\\ 	  -3 & -9  & 6 \\   2  &  6  &-4
	\end{bmatrix} .
	$$ 
	are all singular and
	$\dim (\ker C_1 \cap \ker C_2 \cap \ker C_3 ) = 1.$ 
	This intersection is spanned by $x = (-4,2,1).$

	Consider the linear combination
	$$
	\bC = x C_1 + y C_2+ z C_3 =
	\begin{bmatrix}
	-x - z			&  -4 x - 3 z					&       4 x +  2 z \\
	-4 x - 3 z		& -16 x -  y - 9 z 		&   16 x + 2 y +6 z\\
	4 x +2 z		&    	 16 x + 2 y +6 z 			& - 16 x - 4 y - 4 z
	\end{bmatrix}.
	$$	
	Applying Schm\"{u}dgen's procedure we have
	$$
	\bX_- \bC \bX_-^* 
	= 
	\begin{bmatrix}
	(-x-z)^3 						&0 					 \\
	0									 	&	\bA				  \\
	\end{bmatrix}, \enskip
	\bX_-   =
	\begin{bmatrix}
	- x - z			&     0			&     0\\
	- 4 x - 3 z		& -x - z			&     0\\
	4 x + 2 z 			&     0			& -x - z
	\end{bmatrix} ,
	$$
	where
	\begin{align*}
	\bA &= (-x-z)
	\begin{bmatrix}
	x y + y z +z x  		& 	- 2 (x y + y z + z x  )   \\
	- 2 (x y + y z + z x )       &      4 (x y + y z + z x)
	\end{bmatrix} 
	:= 
	\begin{bmatrix}
	\alpha		& \beta \\
	\beta	    &   \gamma
	\end{bmatrix}.
	\end{align*}
	Let
	$$
	\bX_{i-} = 
	\begin{bmatrix}
	1 & 0 & 0\\
	0 & \alpha & 0\\
	0 & -\beta & \alpha
	\end{bmatrix}.
	$$
	We then have
	\begin{align*}
	\bX_{i-} (\bX_- \bC \bX_-^*) \bX_{i-}^* &= 
	\begin{bmatrix}
	(-x-z)^3 & 0  			& 0\\
	0	& \alpha^3 		& 0\\
	0	& 0				& \alpha(\alpha\gamma - \beta^2)
	\end{bmatrix},
	\end{align*}
	where 
	\begin{align*}
	\alpha &= (-x-z) (x y + y z + z x),\\
	\beta  &= 2(-x-z) (x y + y z + z x),\\
	\gamma &=4(-x-z)(x y + y z + z x).
	\end{align*}
	It is easy to check that
	$ 
	\alpha\gamma - \beta^2 =0
	%= 4(-x-z)^2(xy+yz+zx)^2-4(-x-z)^2(xy+yz+zx)^2 =0,
	$ 
	for all $x,y,z.$
	then we have 
	$
	\alpha(\alpha\gamma - \beta^2)=0.
	$
	
	On the other hand, we have $$X_{i-}X_{i+}=X_{i+}X_{i-}=\begin{bmatrix}
	1 & 0  			& 0\\
	0	& \alpha^2		& 0\\
	0	& 0				& \alpha^2
	\end{bmatrix}, $$
	The procedure stops. We have $r=\rk C(\lambda)=2.$ 
	Since $\cap_{i=1}^{3}\ker C_i=\{x=(-4 a, 2 a, a)/a \in \Bbb R\},$ $\dim (\cap_{i=1}^{3}\ker C_i)= 1.$ 
	
	Pick
	$$
	Q= \begin{bmatrix}
	1 & 0  			& -4\\
	0	& 1 		& 2\\
	4	& -2				& 1
	\end{bmatrix},
	$$
	then
	%such that
	$$ Q^*C_1Q= \begin{bmatrix}
	-225 & 180  			& 0\\
	180	& -144 		& 0\\
	0	& 0				& 0
	\end{bmatrix}, A_1=\begin{bmatrix}
	-225 & 180  		\\
	180	& -144 	
	\end{bmatrix}$$
	$$ Q^*C_2Q=  \begin{bmatrix}
	-64 & 40  			& 0\\
	40	& -25 		& 0\\
	0	& 0				& 0
	\end{bmatrix}, A_2= \begin{bmatrix}
	-64 & 40  		\\
	40	& -25 		
	\end{bmatrix}$$
	$$ Q^*C_3Q=  \begin{bmatrix}
	-49 & 49  			& 0\\
	49	& -49 		& 0\\
	0	& 0				& 0
	\end{bmatrix}, 
	A_3=
	\begin{bmatrix}
	-49 & 49  		\\
	49	& -49 \end{bmatrix}.		
	$$
	We have 
	$$
	A:=A(\lambda)=-A_2+A_3
	=
	\begin{bmatrix}
	15 & 9  		\\
	9	& -24 	
	\end{bmatrix}
	$$
	and $det A(\lambda)=-441 \neq 0.$
	
	On the other hand,  
	$$ 
	A^{-1}A_2
	=
	\left[
	\begin{array}{lll}
	- \dfrac{8}{3} & -\dfrac{5}{3}  		\\
	\ - \dfrac{8}{3}	&\ \ \  \dfrac{5}{3} 	
	\end{array}
	\right];  
	A^{-1}A_1=
	\begin{bmatrix}
	\dfrac{60}{7} & \dfrac{48}{7}  		\\
	\dfrac{75}{7}	& \dfrac{60}{7} 	
	\end{bmatrix};$$
	$$ A^{-1}A_1.A^{-1}A_2=\begin{bmatrix}
	\dfrac{-288}{7} & \dfrac{-20}{3}  		\\
	\dfrac{-360}{7}	& \dfrac{-25}{7} 	
	\end{bmatrix};  A^{-1}A_2.A^{-1}A_1=\begin{bmatrix}
	\dfrac{-285}{7} & \dfrac{-228}{7}  		\\
	\dfrac{-35}{7}	& 4 	
	\end{bmatrix};$$
	
	By Theorem \ref{theo:maxrk}, $A_1, A_2, A_3$ are not SDC. 
	We conclude, therefore, that $C_1,C_2,C_3$ are not SDC.
		\hfill$\diamond$
\end{exam}

\section{Completely solving the Hermitian SDC problem} \label{sec:sdp}

As a consequence of Theorem \ref{theo:H1},
every commuting collection of Hermitian matrices can be SDC.
However,
this is just a sufficient but not necessary condition.
For example, it is shown in \cite{Ngan2020} that the matrices
$$ 
C_1 =
\begin{bmatrix}
-1 & -2 &  0\\ -2  & -28  & 0\\ 0  & 0  & 5
\end{bmatrix}, \enskip
C_2 = 
\begin{bmatrix}
1 &  2 & 0\\ 2 & 20 & 0\\ 0 &  0 & -3
\end{bmatrix}, \enskip
C_3 = 
\begin{bmatrix}
2 &4  &0\\ 4 &1 &0\\ 0 &0 &7
\end{bmatrix}
$$
are SDC by 
$$
P = 
\begin{bmatrix}
1 & 0 & -2\\
0 & 0 &1\\
0 & 1 &0
\end{bmatrix}
$$
but $C_1C_2\neq C_2C_1.$
The following provides some equivalent SDC conditions for Hemitian matrices.
It turns out that 
the SDC property of a collection of such matrices is equivalent to the feasibility of a positive semidefinite program (SDP).
This also allows us to use SDP solvers, for example, ``\texttt{CVX}'' \cite{b470},  \ldots 
to check the SDC property of Hermitian matrices.

We first provide some equivalent conditions one of which, see the condition (iv),  leads to our main Algorithm \ref{alg:sdc} below.

\begin{theo}\label{theo:H2}
 The following conditions are equivalent:
	\begin{enumerate} [\rm (i)]
		\item 
		Matrices $C_1, \ldots, C_m \in \H^n$ are  SDC.
		\item
		There exists a nonsingular matrix $P\in \C^{n\times n} $ such that 
		$P^* C_1 P,  \ldots, P^* C_m P$
		are commuting.
		
		\item
		There exists a positive definite matrix $Q=Q^*\in \H^n $ such that 
		$Q C_1 Q,  \ldots, Q  C_m Q$
		are commuting.

		\item
		There exists a positive definite $X=X^*\in \H^n $ solves the following system of $\frac{m(m+1)}{2}$ linear equations
			\begin{equation} \label{eq:sdp}
			\begin{array}{llll}
%%			\min \tr(X) 										& \\
%%						\mbox{subject to}		 
%												&	   X \succ 0, \\
												& C_i X C_j = C_j X C_i, \enskip \forall 1\leq i< j \leq n.
			\end{array}
			\end{equation}	
	\end{enumerate}  
If $C_1, \ldots, C_m$ are real then so are all other matrices in the theorem.
\end{theo}
\begin{proof}
	(i$\Rightarrow$ii).
	If $C_1, \ldots, C_m$ are  SDC then 
	$P^* C_i P$ is diagonal for every $i=1, \ldots, m,$
	for some nonsingular matrix $P.$
	This yields the later matrices are commuting..

	\noindent
	 (ii$\Rightarrow$iii).
	Applying polar decomposition to $P,$
	$P = QU $ ($Q=Q^*$ positive definite, $U$ unitary),
	we have
	\begin{align*}
	( U^*Q^* C_i Q U) ( U^*Q^* C_j Q U) & = (P^*C_iP)(P^*C_jP)
	= (P^*C_jP)(P^*C_iP) \\
	&= ( U^*Q^* C_j Q U) ( U^*Q^* C_i Q U),
	\end{align*}
	and hence $(QC_iQ) $ and $QC_jQ$ commute.

	\noindent
	 (iii$\Rightarrow$i). 	
	This implication is clear by Theorem \ref{theo:H1}.	
	
	\noindent
	 (iii$\Leftrightarrow$iv).	
The existence of a positive definite matrix $P$  such that
$$
(PC_iP) (PC_jP) = (PC_jP)(PC_iP), \enskip \forall i\neq j,
$$
in other words
$C_iP^2C_j = C_jP^2 C_i,$ %for all $i\neq j,$
is equivalent to which of  the positive definite matrix $X$ satisfying
$C_iXC_j = C_jX C_i,$ for all $i\neq j.$
It is clear that the last equations are just linear in $X.$
%Namely, 
%\eqref{eq:sdp} is a SDP.	
Conversely, if $X$ is positive definite which satisfies 
$C_iXC_j = C_jX C_i,$ for all $i\neq j,$
then $P$ is just picked as the square root of $X.$
\end{proof}

%\red{Change into Corollary as a generalization of [15, Thm. 10]!}
%\begin{lemm} \mbox{}\cite[Thm. 10]{jiang2016} \label{lemm:ji2}
%	Suppose there is $0\neq \lambda \in \R^m$ such that $\bC(\lambda) \succ 0,$ where, without loss of generality, we assume $\lambda_1 \neq 0.$
%	Then
%	$C_1, \ldots, C_m \in \Sy^n$ 
%	is SDC if and only if 
%	$P^\mathrm{T} C_i P$ and $P^\mathrm{T} C_j P$ commute for all 
%	$2 \leq  i\neq j \leq m,$
%	where $P$ is determined such that 
%	$P^\mathrm{T} \bC(\lambda) P= I$ (the identity matrix).
%\end{lemm}
%
%As shown in \cite{jiang2016},
%the matrix $P$ in Lemma \ref{lemm:ji2} 
%is determined as	
%$P = UD^{1/2},$
%where  $U$ is orthogonal and $D^{1/2} $ is the square root of the diagonal matrix $D$ in an eigenvalue decomposition of 
%$\bC(\lambda):$
%$$
%D = U^\mathrm{T} \bC(\lambda) U.
%$$

\subsection{The algorithm}

Based on Theorems \ref{theo:H1} and \ref{theo:H2},
our algorithm 
consists of two stages:

(1) detecting whether the given hermtian matrices are SDC 
 by solving the linear system \eqref{eq:sdp},
and obtaining a commuting  Hermitian matrices;

(2) simultaneously diagonalizing via congruence the resulting matrices.
% by Algorithm \ref{alg:withI}.

It may apply Algorithm \ref{alg:withI} to perform the second stage.
However,
to proceed Step 1 of 
Algorithm \ref{alg:withI}, 
one needs to compute eigenvalue decomposition of all matrices $C_1, \ldots, C_m.$
This may present a high complexity. 
Our main algorithm (Algorithm \ref{alg:sdc}) 
then prefer Algorithm \ref{alg:dodo} below to Algorithm \ref{alg:withI} for the second stage.
It
will exploit the works in \cite{Bustamante2020, Chris2020},
where
the work \cite{BuByMe93}
deals with a simultaneous  diagonalization of two commuting normal matrices 
and
the very recent
work \cite{Chris2020} extends  to more (commuting normal) matrices
with a performance in \textsc{Matlab}.
This extension can be summarized as follows.
Suppose 
$C_i =[c_{u v}^{(i)}] \in \H^n $ 
and let
\begin{subequations}% \label{eq:off}
\begin{align}%{lll}
\mathrm{off}_2
&  = 
\mathrm{off}_2(C_1, \ldots, C_m) = 
\sum_{i=1}^{m} \sum_{u \neq v} |c_{u v}^{(i)}|^2, \label{eq:off1}\\
%=\sum_{i=1}^{m} \left\| ( C_i - \diag(\diag(C_i) )) \right\|_F^2. \\
R(u,v,c,s)
& =  
I_n + (c-1)\mathbf{ e}_u \mathbf{ e}_u^T - \bar{s}\mathbf{ e}_u \mathbf{ e}_v^T + s \mathbf{ e}_v \mathbf{ e}_u^T + (\bar{c}-1) \mathbf{e}_v \mathbf{ e}_v^T, \label{eq:off2}
% c & = \cos \theta \mbox{ and } s = e^{\mathbf{ i} \phi} \sin \theta, \label{eq:off3}
\end{align}
\end{subequations}
where 
$u,v =1, \ldots,n$ and
$c,s \in \C$ with 
$|c|^2+|s|^2 =1.$ 
It can be verified that for 
a given pair $(c,s)$ and every
pair $(u,v) \in \{1, \ldots, n\}^2,$ 
the following holds true:
\begin{align}	\label{eq:offR}
\mathrm{off}_2(RC_1R^*, \ldots, RC_mR^*)
& = \mathrm{off}_2(C_1, \ldots, C_m) - 
\sum_{i=1}^{m} \left( |c_{u v}^{(i)}|^2 + |c_{vu}^{(i)}|^2\right) \notag \\
& \enskip 
    + \sum_{i=1}^{m} 
    	\left| c^2 \bar{c}_{uv}^{(i)} + cs (\bar{c}_{uu}^{(i)} - \bar{c}_{vv}^{(i)}) - s^2 \bar{c}_{vu}^{(i)} \right|^2  \notag \\
& \enskip 
+ \sum_{i=1}^{m} 
	\left |c^2  {c}_{vu}^{(i)} + cs ( {c}_{uu}^{(i)} -  {c}_{vv}^{(i)}) - s^2  {c}_{uv}^{(i)} \right|^2.
\end{align}
In their method \cite{BuByMe93, Chris2020},
at the loop w.r.t each $(u,v),$
it tries to find $c,s$ that makes 
$
\mathrm{off}_2(RC_1R^*, \ldots, RC_mR^*) < \mathrm{off}_2( C_1 , \ldots,  C_m ).
$
It is shown in, eg., \cite{GoHo59}, that $c,s$ can be looked for that minimize 
the last sum on the right hand-side of \eqref{eq:offR}. 
This is equivalent to that solve that minimize the amount
$
%\min_{\|z\|_2=1}\{ 
\| M_{uv} z \| _2
%:\ z \}
$
with
$$
M_{uv} =
\begin{bmatrix}
\bar{c}_{uv}^{(1)} &
( \bar{c}_{uu}^{(1)} - \bar{c}_{vv}^{(1)}) & - \bar{c}_{vu}^{(1)}\\
 {c}_{vu}^{(1)} &
(  {c}_{uu}^{(1)} -  {c}_{vv}^{(1)}) & - {c}_{vu}^{(1)}\\
\vdots & \vdots & \vdots \\
\bar{c}_{uv}^{(m)} &
( \bar{c}_{uu}^{(m)} - \bar{c}_{vv}^{(m)}) & - \bar{c}_{vu}^{(m)}\\
{c}_{vu}^{(m)} &
(  {c}_{uu}^{(m)} -  {c}_{vv}^{(m)}) & - {c}_{vu}^{(m)}
\end{bmatrix}, \enskip
z = 
\begin{bmatrix}
c^2\\ sc \\ s^2
\end{bmatrix}.
$$
Note that $\|z\|_2=1$ and $c,s$ can be parameterized as 
$c =\cos(\theta_{uv}),$
$s = e^{\mathbf{ i} \phi_{uv} }\sin(\theta_{uv}),$
	$(\theta_{uv}, \phi_{uv}) \in [-\frac{\pi}{4}, \frac{\pi}{4}] \times [-\pi, \pi].$ 
 
\noindent
%\rule[-2ex]{\linewidth}{1.0pt}
\begin{algo}   \label{alg:dodo}  
	%\noindent \textbf{\red{Algorithm 3}}. 
	SDC commuting Hermitian matrices.\\ 
%	\rule[1ex]{\linewidth}{0.5pt}
\begin{tabular}{|p{\textwidth} | }
	\hline 
	\textsc{Input}: \hskip3mm
	Commuting Hermitian matrices 
	$C_1, \ldots, C_m \in \C^{n\times n},$
 a tolerance $\epsilon >0.$

	\noindent 
	%\underline{
	\textsc{Output}:
	%} 
	A unitary matrix $U$ such that 
	$\mathrm{off}_2 \leq \epsilon \sum_{i=1}^m \| C_i \|_F:= \nu(\epsilon).$  
	%\rule{\linewidth}{1.5pt}
	%\vskip3mm
	\\ \hline 
	\vspace{.1mm}
	\textit{Step 1}.
		Accumulate  $Q = I_n.$
		
	\textit{Step 2}. 
		\textsc{While} $\mathrm{off}_2 > \nu(\epsilon)$ 
		\begin{enumerate}[(i)]
			\item 
			For every pair $(u,v),$ $1\leq u < v\leq n,$ 
			determine the rotation $R(u,v,c,s)$ with $(c,s) = (\cos \theta_{uv} ,  e^{\mathbf{i} \phi} \sin  \theta_{uv} )$ being the solution to the problem
			$\min\{\| M_{uv} z \|_2:\  z =[c^2\enskip sc \enskip s^2]^T\} $ above.
			\item			
				Accumulate $Q = QR(u,v,c,s),$ $C_i = R(u,v,c,s)^* C_i R(u,v,c,s),$ $i=1, \ldots, m.$
		\end{enumerate}
	\vspace{-7mm}
\\ \hline
\end{tabular}	
%	\vskip-5mm
%	\rule[0ex]{\linewidth}{1.0pt}
\end{algo}

What we have discussed leads to the main algorithm as follows.
\begin{algo}  \label{alg:sdc}  
%\noindent \textbf{\red{Algorithm 3}}. 
 Solving the Hermitian SDC problem.\\ 
\begin{tabular}{|p{\textwidth}|} 
	\hline
\textsc{Input}: \hskip3mm
Hermitian matrices 
$C_1, \ldots, C_m \in \C^{n\times n}$
(not necessary commuting).

\noindent 
%\underline{
\textsc{Output}:
%} 
A nonsingular matrix $U$ such that $U^* C_i UP$'s are diagonal (if exists).
%\rule{\linewidth}{1.5pt}
%\vskip3mm
\\ \hline
\vspace{-4mm}
%$\bullet$ \textsf{At Iteration} $i\geq 1:$	
\begin{enumerate}[\it Step 1:]
\item 
	If the system \eqref{eq:sdp} has no a positive definite solution $P,$
	conclude the initial matrices are not  SDC.
	
	Otherwise, compute the square root $Q$ of $P,$ $Q^2 = P$ and
%	 by using eigenvalue decomposition of $P.$ 
	go to Step 2.
\item 
		
	Apply Algorithm \ref{alg:dodo} to find a unitary matrix $U$ that 
	 simultaneously diagonalizes the matrices $I,$ $ Q^*C_1Q, \ldots, Q^*C_mQ.$ % by a unitary matrix $V.$
	Then $U=QV.$
\end{enumerate}
\vspace{-7mm}
\\ \hline
\end{tabular}
\end{algo}

\begin{rema} 
	In Algorithm \ref{alg:sdc} ,
		the output data will be automatically real if the input is real.
In this situation, the input matrices are all real symmetric. 
	Let $c_{uv}^i$ be the $(u,v)$th entry of $C_i \in \Sy^n.$ 
	A positive definite matrix $X=[x_{uv}]$ that solves \eqref{eq:sdp} will equivalently satisfy
	\begin{equation}\label{eq:sdp-mod}
	X\succ 0, \enskip
	\tr(A_{ij} X) = 0, \enskip 1\leq i<j\leq m,
	\end{equation}
	where $A_{ij} =[a_{uv}^{ij}] \in \Sy^n$ with
	$$
	a_{uv}^{ij} = 
	\left\{
	\begin{array}{lllll}
	\sum_{p,q=1}^n (c_{pu}^i c_{qu}^j -c_{pu}^j c_{qu}^i)	 &  \mbox{if } u=v,\\
	2\sum_{p,q=1}^n (c_{pu}^i c_{qv}^j - c_{pu}^j c_{qv}^i)	 &  \mbox{if } u\neq v.
	\end{array}
	\right.
	$$
	It is well known, see eg., in \cite{KlepSch13, r697}, that 
	$$
	\{X\in \Sy^n_+ :\ \tr(A_{ij} X) =0,\enskip 1\leq i<j\leq m\} \neq \{0\}
	\Leftrightarrow
	\mathtt{span}\{A_{ij}\}_{1\leq i<j\leq m} \cap\ \Sy^n_{++} = \emptyset.
	$$
	We then have the following.
	\hfill $\diamond$
\end{rema}

\begin{coro}
	If $\mathtt{span}\{A_{ij}\}_{1\leq i<j\leq m} \cap\ \Sy^n_{++} \neq  \emptyset$
		then $C_1, \ldots, C_m$ are not SDC. 
\end{coro}

\subsection{Examples} \label{sec:exam2}
 
\begin{exam} \label{ex:1}
The matrices
$$
C_1 = \begin{bmatrix} 0 & 1\\ 1 & 1\end{bmatrix}, \enskip
C_2 = \begin{bmatrix} 1 & 1\\ 1 & 0\end{bmatrix}
$$
are $\C$-SDC as shown in \cite{Bustamante2020}.
However, they are
 not SDC by Theorem \ref{theo:maxrk} 
 since $C_1$ is invertible and
$$
C_1^{-1} C_2 = 
\begin{bmatrix}
 0 & -1\\ 1 & 1
\end{bmatrix}
$$
has only complex eigenvalues $\frac{1\pm i\sqrt{3} }{2}.$

We can check this by applying Theorem \ref{theo:H2} as follows.
The matrices
are SDC if and only if there is a positive semidefinite matrix 
$
X = \begin{bmatrix} x & y\\ y & z\end{bmatrix} \succ 0,
$
which is equivalent to
$x>0$ and $xz > y^2,$
such that
$$
C_1 X C_2 = C_2XC_1  \left( = (C_1 X C_2)^* \right).
$$
This is equivalent to
$$
\left\{
\begin{array}{llll}
x > 0, & xz> y^2 \\
x + y + z & =0.
\end{array}
\right.
$$
But the last condition is impossible since there do not exist $x,z >0$ such that
$
xz > y^2 = (x+z)^2.
$
Thus $C_1$ and $C_2$ are not SDC on $\R.$
%
%However, it is shown in \cite{Bustamante2020} that 
%these are SDC via $^\mathrm{T}$-congruence on $\C.$
%We now apply Algorithm \ref{algo:3}  to prove this.
%Firstly, since $det(C_1) = -1,$
%we find an eigenvalue decomposition of $C_1$ as
%$$
%P^* C_1 P = \diag\left(\frac{1+\sqrt{5}}{2}, -\frac{\sqrt{5} -1}{2} \right), \enskip
%P = 
%\begin{bmatrix}
%\frac{2}{\sqrt{10 +2\sqrt{5}}} 						& \frac{2}{\sqrt{10 -2\sqrt{5}}}\\
%\frac{1+\sqrt{5}}{\sqrt{10 +2\sqrt{5}}} 		& \frac{1-\sqrt{5}}{\sqrt{10 -2\sqrt{5}}} 	
%\end{bmatrix}.
%$$
%We then have $U^* C_1 U = I$ where
%$$
%U =P \diag \left( \frac{\sqrt{2}}{\sqrt{1+ \sqrt{5}}}, i\frac{\sqrt{2}}{\sqrt{ \sqrt{5} -1}} \right).
%$$
%It is clear that 
%$U^* C_1 U = I$ and $U^* C_2U$ commute,
%they are hence (complex) SDC.
\hfill$\diamond$
\end{exam}

\begin{exam} 
	Reconsider the matrices in Example \ref{ex:2}.
	To apply Theorem \ref{theo:H2} or Algorithm \ref{alg:sdc} ,
	we want to find  
		\begin{equation} \label{eq:X}
		X = \begin{bmatrix} x & y & z\\ y & t & u\\ z & u & v		\end{bmatrix} \succ 0 \enskip
		\left(\Leftrightarrow x>0, xt >y^2, \det(X) >0 \right)		
		\end{equation}
		such that 
		$$
		C_1X C_2 = (C_1X C_2)^*, \enskip
		C_1X C_3 = (C_1 X C_3)^*		, \enskip
		C_2 X C_3= (C_2 X C_3)^*.		
		$$
By directly computing,
%{\small
%\begin{align*}
%C_1XC_2 &= 
%\begin{bmatrix}
% 0 	&   12 u - 9 t - 4 v - 3 y + 2 z			&      6 t - 7 u + 2 v + 2 y - z\\
%  0  &  62 u - 48 t - 20 v - 9 y + 6 z 		&	 32 t - 36 u + 10 v + 6 y - 3 z \\
%  0  & 30 t - 38 u + 12 v + 6 y - 4 z		&  22 u - 20 t - 6 v - 4 y + 2 z 
%\end{bmatrix}, \enskip
%\\
%C_1XC_3 & = 
%\begin{bmatrix}
%    12 u - 9 t - 4 v - x - 6 y + 4 z				&   22 u - 15 t - 8 v - 3 x - 14 y + 10 z			&    12 t - 17 u + 6 v + 2 x + 10 y - 7 z \\
%  62 u - 48 t - 20 v - 3 x - 25 y + 16 z 	& 114 u - 80 t - 40 v - 9 x - 63 y + 42 z		& 64 t - 88 u + 30 v + 6 x + 44 y - 29 z \\
%  30 t - 38 u + 12 v + 2 x + 16 y - 10 z	&  50 t - 70 u + 24 v + 6 x + 40 y - 26 z		& 54 u - 40 t - 18 v - 4 x - 28 y + 18 z \\
%\end{bmatrix}, \enskip
%\\
%C_2XC_3 &=
%\begin{bmatrix}
%                             0                            & 0                              					& 0 \\
%  9 t - 12 u + 4 v + 3 y - 2 z						 & 15 t - 22 u + 8 v + 9 y - 6 z		& 17 u - 12 t - 6 v - 6 y + 4 z  \\
%  7 u - 6 t - 2 v - 2 y + z							& 13 u - 10 t - 4 v - 6 y + 3 z					&  8 t - 10 u + 3 v + 4 y - 2 z
%\end{bmatrix}.
%\end{align*}
%}
%
$$
C_1 X C_2 = (C_1 X C_2)^* \Leftrightarrow
\left\{
\begin{array}{lllllll}
12 u  &- 9 t		& - 4 v		& - 3 y 		&+ 2 z	 		& =0\\
- 7 u  & +6 t  &+ 2 v 		&+ 2 y 		&- z				& = 0\\
2 u 		&+2 t 	& -2 v 		& 				&+  z				& =0,
\end{array}
\right. 
$$
$$
C_1 X C_3 = (C_1 X C_3)^* \Leftrightarrow
\left\{
\begin{array}{lllllll}
40u 	&- 33t 		&- 12v 		&- 11 y			& + 6 z	 		& =0\\
7 u 		&- 6 t 		&- 2 v 		&- 2 y 			&+  z			& = 0\\
18  u 	&- 14 t 		&- 6 v 		&- 4 y 			&+ 3 z			& =0,
\end{array}
\right. 
$$
$$
C_2 X C_3 = (C_2  X C_3)^* \Leftrightarrow
\left\{
\begin{array}{llllllll}
  - 12 u  &+ 9 t  	&+ 4 v  & + 3 y  & - 2 z						 & =0		\\
  4 u 		&- 2 t		& - 2 v		&			& +  z	  					&=0\\
  7 u		& - 6 t 		&- 2 v 		&- 2 y 	&+ z							& 					=0.
\end{array}
\right.
$$
Combining the linear equations above we obtain
$$
u=2y, \enskip
t = y, \enskip
v = 3y+  \frac{z}{2}.
$$
We then pick
$y=1, z=4,x=6$ 
and
$$
X = 
\begin{bmatrix}
6 & 1 & 4\\
1 & 1 & 2\\
4 & 2 & 5
\end{bmatrix} \succ 0
$$
makes 
$
\sqrt{X} C_1 \sqrt{X},  \sqrt{X} C_2 \sqrt{X}, \sqrt{X} C_3 \sqrt{X}
$
to be commuting
by Theorem \ref{theo:H2}.
Thus three initial matrices are SDC on $\R,$ and so are they on $\C.$ 
\hfill$\diamond$
\end{exam}

We now consider other examples of which all given matrices are singular.

 \begin{exam} 
 	For the matrices in Example \ref{ex:3},
we will check if there exists $X$ as in \eqref{eq:X} satisfying the following:
$$
C_1 X C_2 = (C_1 X C_2)^* \Leftrightarrow
\left\{
\begin{array}{lllllll}
9 u  &- 9 t		& - 2 v		& - 3 y 		&+ 2 z	 		& =0\\
- 5 u  & +6 t  &+   v 		&+ 2 y 		&- z				& = 0\\
-12 u 		&+12 t 	& +4 v 		& +3y				&-  z				& =0,
\end{array}
\right. 
$$
$$
C_1 X C_3 = (C_1 X C_3)^* \Leftrightarrow
\left\{
\begin{array}{lllllll}
-5u 	&- 3t 		& +4v 	 &		&- y			& - z	 		& =0\\
19 u 	&- 12 t 	&- 7 v 	 & -x	&- 7 y 			&+  5z			& = 0\\
28  u 	&- 24 t 	&- 8 v 	 &-3x	&- 19 y 		&+ 11 z			& =0,
\end{array}
\right. 
$$
$$
C_2 X C_3 = (C_2  X C_3)^* \Leftrightarrow
\left\{
\begin{array}{llllllll}
- 12 u  &+ 9 t  	&+ 4 v  & + 3 y  & - 2 z		 & =0		\\
4 u 		&- 2 t		& - 2 v		&			& +  z	  					&=0\\
7 u		& - 6 t 		&- 2 v 		&- 2 y 	&+ z							& 					=0.
\end{array}
\right.
$$
The system of linear equations combined by three conditions above has only trivial solution (since its argument matrix is full rank). 
This means only zero matrix  satisfies \eqref{eq:sdp}.
The given matrices are hence not SDC on $\R.$
\hfill$\diamond$
\end{exam}

\begin{exam} 
Let us turn to the matrices in Example \ref{ex:4}.
We will check if there exists $X$ as in \eqref{eq:X} satisfying the following:
$$
C_1 X C_2 = (C_1 X C_2)^* \Leftrightarrow
\left\{
\begin{array}{lllllll}
y  &- 2 z		& + 4 t		& - 12 u 		&+ 8 v	 		& =0\\
- 2 y  & +4 z  &-   8 t 		&+ 24 u 		&- 16 v				& = 0\\
4 y 		&-8 z 	& +16 t 		& -48 u				&+ 32 v				& =0,
\end{array}
\right. 
$$

$$
C_1 X C_3 = (C_1 X C_3)^* \Leftrightarrow
\left\{
\begin{array}{lllllll}
x 	&+ 7y 		& -6 z 	 &+ 12 t		&- 20 u			& + 8 v	 		& =0\\
2 x 	&+ 14 y 	&- 12 z 	 & + 24 t	&- 40 u 			&+  16v			& = 0\\
4  x 	&+ 28 y 	&- 24 z	 &+ 48 t	&- 80 u 		&+ 32 v			& =0,
\end{array}
\right. 
$$

$$
C_2 X C_3 = (C_2  X C_3)^* \Leftrightarrow
\left\{
\begin{array}{llllllll}
y  &- 2 z  	&+ 3 t  & - 8 u  & + 4 v		 & =0		\\
- 2 y 		&+ 4 z		& - 6 t		&	+ 16 u		& -  8 v	  					&=0\\
4 y		& - 8 z 		&+12 t 		&- 32 u 	&+ 16 v							& 					=0.
\end{array}
\right.
$$
The general  solutions of these linear equations are of the form
$$
\left\{
\begin{array}{llllllll}
x &=  	&- 4 y       	& - 16 u  	& 	+16 v	\\
t &= 	& 	    		& \quad 4 u		&	-4 v	\\
z & = 	&\dfrac{y}{2} 	&+  2 u 	&- 4 v. 			
\end{array}
\right.
$$
We have $$ xt = -16y(u-v)-64(u-v)^2>y^2  \Leftrightarrow (y+8u-8v)^2<0$$
There do not exist $x,t$ such that $xt>y^2.$ This means there is no positive definite matrix  satisfying \eqref{eq:sdp}.We  have $C_1,C_2$ and $C_3$ are not $\R$-SDC.
\hfill$\diamond$
\end{exam}	

\subsection{Numerical tests}
In this section we give numerical tests illustrating Algorithm \ref{alg:sdc} .

There are several methods for computing a square root of a positive definite matrix. 
In our \textsc{Matlab} experiments, we exploit the \textsc{Matlab} function
``\texttt{sqrtm.m}'' to find a square root of matrix $Q$  in Step 1.
This function executes the algorithm provided in \cite{Deadman2013} for computing a square root of an arbitrary square matrix.
Table \ref{table1} shows some numerical tests with respect to several values of $m$ and $n.$ 
Each result is the average of three ones.
The Hermitian matrices $C_1, \ldots, C_m,$
of which the SDC property will afterwards be estimated, 
are constructed as 
$C_i = P^* D_i P,$
where
  $D_i$ is diagonal
  and $P$ is invertible 
  that are
  randomly taken from a uniform distribution on the interval $[0, 1).$
The backward errors are estimated as 
$$
\mathrm{Err} = \max \left\{ \dfrac{\| U^* C_i  U - \diag ( \diag(U^* C_i  U) )  \|_2}{\| U^* C_i  U \|_2}  \bigg|  i=1, \ldots, m \right\},
$$
where $\diag ( \diag(X) )$ denote the diagonal matrix whose diagonal 
is of $X.$
For the Stage 2 w.r.t. Algorithm \ref{alg:dodo}, we fix a tolerance to be the floating-point relative accuracy ``\texttt{eps}'' of \textsc{Matlab}, to the power of $\frac{3}{2}.$

\begin{table}[h]
\centering{
	\begin{tabular}{| l | l | c | c |}
		\hline
		$m$       & $ n$ & Err 				& CPU time (s) \\ \hline
		3        & 3  &     3.33e-12 & 6.59 \\ \hline
		10 		& 20 & 8.64e-13		& 922.81 \\ 
		50 		& 100 &  				&   \\ 
		50 		& 200 &  				&   \\ 
		100 		& 100 &  				&   \\ 				
		\hline
	\end{tabular}
}
	\caption{Numerical tests the SDC property of collections of Hermitian matrices.} \label{table1}
\end{table}

\section{Conclusion and discussion}
We have provided some equivalent conditions detecting whether  a collection of Hermitian matrices can be simultaneously diagonalized via $*$-congruence.
One of these conditions 
leads to solve a positive semidefinite program with a positive definite solution.
Combining this with an existing Jacobi-like method for the simultaneous diagonalization of commuting normal matrices, we propose an algorithm for 
detecting/computing a simultaneous diagonalization of a collection of non-commuting Hermitian matrices.
We have also present some numerical tests for this main algorithm.

However, not every collection of normal matrices can be SDC.
In some applications, for example the quadratically constrained quadratic programming,
one may seek to approximate given matrices by simultaneously diagonalizable (via congruence) ones.
One should addressed two possibilities for approximation in the future work:
the first one make the initial matrices become pairwise commute, and other is that of immediately finding SDC matrices that ``colse to'' the initial ones.

\subsubsection*{Acknowledgement}  
%XXX
%The first author would like to thank Vietnam Institute for Advanced Study in
%Mathematics (VIASM) for warm hospitality
%during his visit.
%
%
%The first author %This work 
%is partly supported by 
%Vietnam National Foundation for Science and Technology Development (NAFOSTED) under grant number 
%101.04-2017.312.

\appendix

\section{}%{The SDC equivalence between matrices and their submatrices}
 \label{app:1}

\begin{lemm}\label{lem:appdx1} 
	The matrices
$
C_1= 
	\begin{bmatrix}
	0_k & 0 \\ 0 & \hat{C}_1
	\end{bmatrix},
\ldots,
C_m=
	\begin{bmatrix}
0_k & 0 \\ 0 & \hat{C}_m
\end{bmatrix}
$
are SDC if and only if 
so are 	
$ \hat{C}_1, \ldots, \hat{C}_m.$ 
\end{lemm}
\begin{proof}
If $\hat{C}_1, \ldots, \hat{C}_m$
are SDC by a nonsingular matrix $\hat{P}$
then 
${C}_1, \ldots, {C}_m$
are certainly SDC by the nonsingular matrix
$
P = 
\begin{bmatrix}
I_k & 0 \\ 0 &\hat{P}
\end{bmatrix}.
$

Conversely, 
suppose $C_1, \ldots, C_m$ are SDC by a nonsingular matrix $U.$
Partitioning 
$U =  \left[
\begin{array}{cc}
U_1 &U_2 \\
U_3&U_4 \\
\end{array} 
\right],
$
$U_3 \in \F^{k\times k},$
$U_4 \in \F^{(n-k) \times (n-k)}.$
Note that $\rk [U_3 \enskip U_4] = n-k$ due to the nonsingularity of $U.$
For every $i=1, \ldots, m,$ the following matrix 
$$
U^* 
	\begin{bmatrix}
0_k & 0 \\ 0 & \hat{C}_i
\end{bmatrix}
U
= 
	\begin{bmatrix}
U_3^*\hat{C}_i U_3& U_3^*\hat{C}_i U_4 \\ U_4^*\hat{C}_i U_3 & U_4^*\hat{C}_i U_4
\end{bmatrix}
$$
is diagonal.
Since the later block matrix is diagonal,
we can assume $U_4$ is nonsingular
after multiplying on the right of $U$ by an appropriate permutation matrix.
This means $U_4^*\hat{C}_i U_4$ 
is diagonal for every $i=1, \ldots, m.$
\end{proof}

%====================================
\begin{lemm}\label{lem:rk} 
	Let $C_1, \ldots, C_m \in \H^n$
	and denote 
	$\bC (\lambda) = \lambda_1 C_1 + \ldots + \lambda_m C_m,$ 
	$\lambda=(\lambda_1, \ldots, \lambda_m) \in \R^m.$ 
	%	(it is also called a pencil).
	Then
	\begin{enumerate}[\rm (i)]
		\item
		$
		\bigcap_{\lambda\in \R^m} \ker \bC(\lambda) = \bigcap_{i=1}^m \ker C_i
		= \ker C,
		$
		where
		$C = \begin{bmatrix}C_1^* & \ldots & C_m^* 	\end{bmatrix}^*.$ 
		\item
		$\max\{\rk_{\R} \bC(\lambda) |\ \lambda \in \R^m\} \leq \rk_\R C.$ %[C_1 \ldots C_m].$
		
		\item
		Suppose $\dim_{\R} (\bigcap_{i=1}^m \ker C_i ) =k.$
		%$r= \max\{\rk_{\F} \bC(\lambda)|\ \lambda \in \F^m\} = \rk_{\F} \bC(\underline{\lambda}).$ 
		Then
		$
		\bigcap_{i=1}^m \ker C_i  = \ker \bC(\underline{\lambda})
		$
		for some $\underline{\lambda} \in \R^{m}$
		if and only if 
		$\rk_{\R}\bC(\underline{\lambda}) = \max_{\lambda \in \R^m} \rk_{\R} \bC(\lambda)   = \rk_{\R} C = n-k.
		$		
		%	In this situation, we have
		%	$r=\rk_{\F} C.$% [C_1 \ldots C_m].$
	\end{enumerate}
\end{lemm}
\begin{proof}
	The part (i) is easy to check. 
	The part (ii) is followed from the fact that
	$$
	\rk_{\R} \bC(\lambda) = \rk_{\R} \left(
	\begin{bmatrix} \lambda_1 I & \ldots & \lambda_m I\end{bmatrix} 
	\begin{bmatrix} C_1 \\ \vdots \\ C_m\end{bmatrix} 
	\right)
	\leq \rk_{\R} \begin{bmatrix} C_1 \\ \vdots \\ C_m\end{bmatrix}  = \rk_{\R} C, 
	$$ 
	for all
	$ \lambda \in \R^m.$
	
	For the last part, with the help of the part (i), we have
	$\ker C = \bigcap_{i=1}^m \ker C_i  \subseteq \ker \bC(\underline{\lambda}) .$
	Then by the part (ii), 
	\begin{align*}
	\bigcap_{i=1}^m \ker C_i  = \ker \bC(\underline{\lambda}) 
	& \Longleftrightarrow
	\dim_{\F} \left(\ker \bC(\underline{\lambda}) \right) = \dim_{\F} \left( \bigcap_{i=1}^m \ker C_i \right)   = n- \rk_{\F} C \\
	& \Longleftrightarrow
	\rk_{\F} \bC(\underline{\lambda}) = \rk_{\F} C = n-k \geq \rk_{\F} \bC(\lambda), \forall \lambda \in \F^m.	
	\end{align*}
	This is certainly equivalent to
	$
	n-k = \rk_{\F} \bC(\underline{\lambda}) = \max_{\lambda \in \F^m} \rk_{\F} \bC(\lambda).  % \dim_{\F} \left( \bigcap_{i=1}^m \ker C_i \right) = n - \rk_{\F}[C_1 \ldots C_m].
	$
\end{proof}

%\red{
%	Lemma \ref{lem:rk} XXX
%}

%====================================
		%\bibliographystyle{plain}
		% \bibliography{sdc}

\end{document}